\documentclass[12pt]{amsart}
\usepackage{amssymb,amsfonts,amsthm,amscd,stmaryrd,dsfont,esint,upgreek}
\usepackage[numbers]{natbib}
\usepackage[mathcal]{euscript}

\usepackage[body={6.5in,8in}, top=1.5in,left=1in]{geometry}

\theoremstyle{plain}
\newtheorem{thm}{Theorem}
\newtheorem{cor}{Corollary}
\newtheorem{lem}[cor]{Lemma}
\newtheorem{prop}[cor]{Proposition}

\theoremstyle{definition}
\newtheorem{definition}[cor]{Definition}
\newtheorem{remark}[cor]{Remark}

\numberwithin{cor}{section}
\numberwithin{equation}{section}

\DeclareMathOperator{\tr}{tr}

\DeclareMathOperator{\dist}{dist}
\DeclareMathOperator*{\osc}{osc}

\DeclareMathOperator*{\argmin}{argmin}
\DeclareMathOperator*{\esssup}{ess\,sup}
\DeclareMathOperator*{\essinf}{ess\,inf}
\DeclareMathOperator{\USC}{USC}
\DeclareMathOperator{\LSC}{LSC}
\DeclareMathOperator{\BUC}{BUC}
\DeclareMathOperator{\Var}{Var}
\DeclareMathOperator{\divg}{div}

\renewcommand{\d}{d} % the ambient dimension
\newcommand{\m}{m} % the number of equations in our system
\newcommand{\ep}{\varepsilon}
\newcommand{\R}{\ensuremath{\mathds{R}}}
\newcommand{\Q}{\ensuremath{\mathds{Q}}}
\newcommand{\Rd}{\ensuremath{{\mathds{R}^\d}}}

\newcommand{\M}{\ensuremath{\mathcal{M}}}

\newcommand{\Sy}{\ensuremath{\mathcal{S}^\d}}
\newcommand{\Prob}{\ensuremath{\mathds{P}}}
\newcommand{\E}{\ensuremath{\mathds{E}}}
\newcommand{\Ls}{\ensuremath{\mathcal{L}}}

\newcommand{\iden}{I_\d}
\newcommand{\indx}{\ensuremath{\{1,\ldots,\m \}}}

\begin{document}

\title[Concentration phenomena for neutronic multigroup diffusion]{Concentration phenomena for neutronic multigroup diffusion in random environments}

\author[S. N. Armstrong]{Scott N. Armstrong}
\address{Department of Mathematics\\ University of Wisconsin, Madison\\ 480 Lincoln Drive\\
Madison, Wisconsin 53706.}
\email{armstron@math.wisc.edu}
\author[P. E. Souganidis]{Panagiotis E. Souganidis}
\address{Department of Mathematics\\ The University of Chicago\\ 5734 S. University Avenue
Chicago, Illinois 60637.}
\email{souganidis@math.uchicago.edu}

\date{\today}

\begin{abstract}
We study the asymptotic behavior of the principal eigenvalue of a weakly coupled, cooperative linear elliptic system in a stationary ergodic heterogeneous medium. The system arises as the so-called \emph{multigroup diffusion model} for neutron flux in nuclear reactor cores, the principal eigenvalue determining the criticality of the reactor in a stationary state. Such systems have been well-studied in recent years in the periodic setting, and the purpose of this work is to obtain results in random media. Our approach connects the linear eigenvalue problem to a system of quasilinear viscous Hamilton-Jacobi equations. By homogenizing the latter, we characterize the asymptotic behavior of the eigenvalue of the linear problem and exhibit some concentration behavior of the eigenfunctions.
\end{abstract}

\maketitle

\section{Introduction}
\label{IN}

We study the behavior, as $\ep \to 0$, of the principal eigenvalue and eigenfunction of the weakly-coupled, cooperative elliptic system
\begin{multline} \label{introeq}
-\ep^2 \tr \left( A_\alpha\left(\frac x\ep ,\omega \right) D^2\varphi_\alpha^\ep \right) + \ep b_\alpha\left(\frac x\ep,\omega\right) \cdot D\varphi_\alpha^\ep + \sum_{\beta=1}^\m c_{\alpha\beta}\left(\frac x\ep, \omega\right) \varphi_\beta^\ep \\ = \lambda_\ep \sum_{\beta=1}^\m \sigma_{\alpha\beta} \left(\frac x\ep, \omega\right) \varphi^\ep_\beta\quad \mbox{in} \ U \quad  (\alpha=1,\ldots, \m),
\end{multline}
subject to the conditions
\begin{equation} \label{introbc}
\varphi_\alpha^\ep > 0 \quad\mbox{in} \ U \quad \mbox{and} \quad \varphi_\alpha^\ep = 0 \quad \mbox{on} \ \partial U \quad (\alpha = 1, \ldots, \m).
\end{equation}
Here $\m \geq 1$ is a positive integer and $U\subseteq \Rd$ is a bounded domain. The unknowns are the eigenvalue $\lambda_\ep = \lambda_\ep(\omega,U)$ and the eigenfunctions $\left(\varphi^\ep_\alpha(\cdot,\omega)\right)_{1 \leq \alpha \leq \m}$. The underlying \emph{random environment} is described by a probability space $(\Omega, \mathcal F, \mathds P)$, and the coefficients $A_\alpha$, $b_\alpha$, $c_{\alpha\beta}$ and $\sigma_{\alpha\beta}$ are functions on $\Rd \times \Omega$ which are required to be \emph{stationary} and \emph{ergodic}. (Precise hypotheses are found in Section~\ref{prelim} below.)

The expectation is that large amplitude, high-frequency oscillations persist as $\ep \to 0$, and the goal is to describe these oscillations.

\medskip

Problem \eqref{introeq} has been proposed and extensively studied in periodic media by physicists as a simplified model for the neutron flux in nuclear reactor cores, see \cite{D,L1,L2,LW,WW}. The modeling assumption is that neutrons are moving in the reactor core (the domain $U$) in $\m$ distinct energy groups, each group consisting of neutrons with a similar amount of kinetic energy. The function $\varphi^\ep_\alpha$ is  the steady-state distribution of neutrons in the $\alpha$th energy group inside the core. The matrix $A_\alpha$ describes the \emph{diffusion} of the neutrons in the $\alpha$th group, the vector $b_\alpha$ is the \emph{drift}, $c_{\alpha\beta}$ is the \emph{total cross section}, which represents the interaction of neutrons in various energy groups, and $\sigma_{\alpha\beta}$ models the creation of neutrons by nuclear fission. The factors $\ep^2$ and $\ep$ appear in front of the diffusion and drift terms, respectively, due to a physical assumption that the order of the diffusion and drift should be the same as that of the microscopic lattice. The principal eigenvalue $\lambda_\ep$, in particular whether $\lambda_\ep$ is greater or less than $1$, determines the \emph{criticality} of the reactor. Hence characterizing the asymptotic behavior of $\lambda_\ep$ is of particular importance. We remark that while the model is typically written in divergence form, if the matrices $A_\alpha$ are sufficiently regular, it can be recast in the form of \eqref{introeq}.

\medskip

A very complete mathematical analysis of \eqref{introeq}-\eqref{introbc} in the case of periodic coefficients was performed by Capdeboscq \cite{C2} (see also \cite{P,C1,AC,ACPS,AB1,AB2,AM,PS,APP}). It was shown in~\cite{C2} that the eigenvalue $\lambda_\ep$ admits the expansion
\begin{equation} \label{lamexp}
\lambda_\ep = \overline \lambda + \ep^2 \mu + o(\ep^2) \quad \mbox{as} \ \ep \to 0,
\end{equation}
and the eigenfunctions can be factored as
\begin{equation} \label{phiexp}
\varphi^\ep_\alpha = \overline \psi_\alpha\!\left( \frac x \ep \right) \exp\!\left( - \overline \theta \cdot \frac x\ep \right) \left( u(x) + o(1) \right) \quad \mbox{as} \ \ep \to 0.
\end{equation}
Here $\overline \lambda\in \R$, $\overline \theta \in \Rd$, and the periodic function $\overline \psi_\alpha = \overline \psi_\alpha(y)$ are identified via an optimization of a periodic (cell) eigenvalue problem, while $(\mu,u)$ is the solution of an effective ``recentered" principal eigenvalue problem in the macroscopic domain $U$. Observe that the oscillations of the coefficients on the microscopic scale $\ep$ not only induce oscillations in the solution on a scale of $\ep$, but also produce a large macroscopic effect, namely an exponential drift. 

\medskip

The random setting is different. As we will see later, we cannot expect \eqref{lamexp} and \eqref{phiexp} to hold in full. Moreover, the approach of \cite{C2} does not seem to yield itself to the analysis of \eqref{introeq} in random environments. Instead, we present an alternative approach. The classical Hopf-Cole transformation converts \eqref{introeq} into a quasilinear (viscous) Hamilton-Jacobi system, and we observe that this nonlinear system may be analyzed by the methods recently introduced by Lions and Souganidis \cite{LS3} and developed further by the authors~\cite{AS}. Our main result is the assertion that there exists a deterministic constant $\overline \lambda$ such that, as $\ep \to 0$, $\lambda_\ep \rightarrow \overline \lambda$ almost surely in $\omega$, together with a characterization of $\overline \lambda$. In fact,  $\overline \lambda$ is identified in terms of a convex \emph{effective Hamiltonian} $\overline H$. Furthermore, we exhibit concentration behavior for the eigenfunctions $\varphi^\ep_\alpha$, showing that, under some additional hypotheses on the random environment, we have, as $\ep \to 0$,
\begin{equation*}
-\ep \log \varphi^\ep_\alpha \rightarrow \overline \theta\cdot x \quad \mbox{locally uniformly in} \ U \ \ \mbox{and a.s. in} \ \omega,
\end{equation*}
where $\overline \theta := \argmin_{\Rd} \overline H$. We thereby justify, in random environments, the leading term in \eqref{lamexp} and, under stronger assumptions, in \eqref{phiexp}.

\medskip

The article is organized as follows. In Section~\ref{prelim} we state the assumptions and some preliminary results needed in the sequel. The main results are presented in Section~\ref{MR}. In Section~\ref{H} we prepare the homogenization of the transformed nonlinear system by studying an auxiliary ``cell" problem. In Section~\ref{Hbar}, we define $\overline H$ and do most of the work for the proof of Theorem~\ref{HJthm}, which is given in Section~\ref{SHJthm}. This analysis is applied to the linear system in Section~\ref{proof}, where we prove Theorem~\ref{MAIN}. In Section~\ref{SC}, we show that $\overline H$ is strictly convex in $p$ (and therefore the eigenfunctions concentrate) in uniquely ergodic environments.

\section{Preliminaries} \label{prelim}

\subsection{Notation}

The symbols $C$ and $c$ denote positive constants, which may vary from line to line and, unless otherwise indicated, do not depend on $\omega$. We work in the $\d$-dimensional Euclidean space $\Rd$ with $\d \geq 1$, and we write $\R_+:=(0,\infty)$. The set of rational numbers is denoted by $\Q$. The set of $n$-by-$\d$ matrices is denoted by $\M^{n\times \d}$, and $\Sy \subseteq \M^{\d\times \d}$ is the set of $\d$-by-$\d$ symmetric matrices. If $v,w\in \Rd$, then $v\otimes w \in \Sy$ is the symmetric tensor product which is the matrix with entries $\frac12(v_iw_j+ v_jw_i)$. For $y \in \Rd$, we denote the Euclidean norm of $y$ by $|y|$, while if $M\in \M^{n\times \d}$, $M^t$ is the transpose of $M$. If $M\in\M^{\d\times\d}$, then $\tr(M)$ is the trace of $M$, and we write $|M| := \tr(M^t M)^{1/2}$. The identity matrix is $\iden$. If $U \subseteq\Rd$, then $|U|$ is the Lebesgue measure of $U$. Open balls are written $B(y,r): = \{ x\in \Rd : |x-y| < r\}$, and we set $B_r : = B(0,r)$. The distance between two subsets $U,V\subseteq \Rd$ is denoted by $\dist(U,V) = \inf\{ |x-y|: x\in U, \, y\in V\}$. If $U\subseteq \Rd$ is open, then $\USC(U)$, $\LSC(U)$ and $\BUC(U)$ are respectively the sets of upper semicontinuous, lower semicontinuous and bounded and uniformly continuous functions $U\to \R$. If $f:U\to \R$ is integrable, then we use the notation
\begin{equation*}
\fint_U f \, dy = \frac{1}{|U|} \int_U f \, dy.
\end{equation*}
If $f:U \to \R$ is measurable, then we set $\osc_U f:= \esssup_U f - \essinf_U f$. The Borel $\sigma$-field on $\Rd$ is denoted by $\mathcal{B}(\Rd)$. If $s,t\in \R$, we write $s\wedge t : = \min\{ s,t\}$.

We emphasize that, throughout the paper, all differential inequalities involving functions not known to be smooth are assumed to be satisfied in the viscosity sense. Finally, we abbreviate the phrase \emph{almost surely in} $\omega$ by ``a.s. in $\omega$."

\subsection{The random medium}
The random environment is described by a probability space $(\Omega, \mathcal F, \mathds P)$, and a particular ``medium" is an element $\omega \in\Omega$. We endow the probability space with an ergodic group $(\tau_y)_{y\in \Rd}$ of $\mathcal F$-measurable, measure-preserving transformations $\tau_y:\Omega\to \Omega$. Here \emph{ergodic} means that, if $D\subseteq \Omega$ is such that $\tau_z(D) = D$ for every $z\in \Rd$, then either $\Prob[D] = 0$ or $\Prob[D] = 1$. An $\mathcal F$-measurable function $f$ on $\Rd \times \Omega$ is said to be \emph{stationary} if the law of $f(y,\cdot)$ is independent of $y$. This is quantified in terms of $\tau$ by the requirement that
\begin{equation*}
f(y,\tau_z \omega) = f(y+z,\omega) \quad \mbox{for every} \ y,z\in \Rd.
\end{equation*}
Notice that if $\phi:\Omega \to S$ is a random process, then $\tilde \phi(y,\omega) : = \phi(\tau_y\omega)$ is stationary. Conversely, if $f$ is a stationary function on $\Rd \times \Omega$, then $f(y,\omega) = f(0,\tau_y\omega)$. 

The expectation of a random variable $f$ with respect to $\mathds P$ is written $\E f$, and we denote the variance of $f$ by $\Var(f): = \E(f^2) - (\E f)^2$. If $E \in \mathcal F$, then $\mathds{1}_E$ is the indicator random variable for $E$; i.e., $\mathds{1}_E(\omega) = 1$ if $\omega\in E$, and $\mathds{1}_E(\omega) = 0$ otherwise. 

\medskip

We rely on the following multiparameter ergodic theorem, a proof of which can be found in Becker~\cite{Be}.
\begin{prop} \label{ergthm}
Suppose that $f:\Rd \times \Omega \to \R$ is stationary and $\E |f(0,\cdot)| < \infty$. Then there is a subset $\widetilde \Omega \subseteq \Omega$ of full probability such that, for each bounded domain $V \subseteq \Rd$ and $\omega \in \widetilde \Omega$, 
\begin{equation} \label{erglim}
\lim_{t\to \infty} \fint_{tV} f(y,\omega) \, dy = \E f.
\end{equation}
\end{prop}

\subsection{Assumptions}
The following hypotheses are in force throughout this article. The coefficients 
\begin{gather*}
A_\alpha : \Rd \times \Omega \to \Sy, \quad b_\alpha : \Rd \times \Omega \to \Rd \quad \mbox{and}  \quad c_{\alpha\beta}, \ \sigma_{\alpha\beta} : \Rd \times \Omega \to \R.
\end{gather*}
are measurable and we require that 
\begin{equation} \label{stationary}
A_\alpha,\ b_\alpha, \ c_{\alpha\beta}, \ \mbox{and} \ \  \sigma_{\alpha\beta} \quad \mbox{are stationary for each} \ \ 1\leq \alpha,\beta\leq \m.
\end{equation}
We assume that, for each $\omega\in \Omega$ and $1\leq \alpha\leq \m$,
\begin{equation} \label{AC11}
A_\alpha(\cdot, \omega) \in C^{1,1}_{\mathrm{loc}}(\Rd;\Sy)
\end{equation}
and that $A_\alpha$ has the form $A  = \Sigma_\alpha \Sigma_\alpha^t$, where, for some $C> 0$,
\begin{equation}\label{Alip}
\Sigma_\alpha(y, \omega) \in \M^{\d \times m} \quad \mbox{satisfies} \quad \| \Sigma_\alpha(\cdot,\omega) \|_{C^{0,1}(\Rd)} \leq C.
\end{equation}
We assume that there exists $C> 0$ such that, for all $\omega \in \Omega$ and $1\leq \alpha,\beta \leq \m$,
\begin{equation} \label{bounds}
\| A_\alpha(\cdot,\omega)\|_{C^{0,1}(\Rd)} + \| b_\alpha(\cdot,\omega) \|_{C^{0,1}(\Rd)}  + \| \sigma_{\alpha\beta} (\cdot,\omega) \|_{C^{0,1}(\Rd)} + \| c_{\alpha\beta} (\cdot,\omega) \|_{C^{0,1}(\Rd)} \leq C.
\end{equation}
The matrices $A_\alpha$ are uniformly positive definite in the sense that there exist positive constants $0 < \lambda \leq \Lambda$ such that, for every $y,\xi \in  \Rd$, $\omega\in \Omega$ and index $\alpha$,
\begin{equation}\label{ellip}
\lambda|\xi|^2 \leq A_\alpha(y,\omega) \xi \cdot \xi \leq \Lambda |\xi|^2.
\end{equation}
The matrix $c_{\alpha\beta}$ is \emph{diagonally dominant}, i.e.,
\begin{equation} \label{diagdom}
c_{\alpha\beta} \leq 0 \quad \mbox{for} \ \alpha\neq \beta, \quad \mbox{and} \quad \sum_{\beta=1}^\m c_{\alpha\beta} \geq 0,
\end{equation}
as well as \emph{fully coupled} in the sense that there exists $c> 0$ such that
\begin{equation} \label{coupled}
\left\{ \begin{aligned}
& \mbox{if} \  \{ \mathcal I, \mathcal J \} \ \mbox{is a nontrivial partition of} \ \indx, \ \mbox{then for every} \ y\in \Rd \ \mbox{and} \ \omega\in \Omega,    \\
& \mbox{there exists} \ \alpha\in \mathcal I \  \mbox{and} \  \beta\in \mathcal J \ \mbox{such that} \ c_{\alpha\beta}(y,\omega) \leq -c.
\end{aligned} \right.
\end{equation}
The hypothesis \eqref{coupled} is satisfied, for example, if $c_{\alpha,\alpha+1} \leq -c < 0$ for each $\alpha\in \indx$. Finally, we suppose that, for every $y\in  \Rd$, $\omega\in \Omega$, and $\alpha,\beta=1,\ldots,\m$,
\begin{equation} \label{sigma}
\sigma_{\alpha\beta} \geq 0 \quad \mbox{and} \quad \sum_{\gamma=1}^k \sigma_{\alpha\gamma} \geq c > 0.
\end{equation}
We emphasize that \eqref{stationary}-\eqref{sigma} are assumed to hold throughout this article.

\subsection{Further notation}
It is convenient to write the system \eqref{introeq} in a more compact form. For each $\alpha \in \indx$, let $\Ls_\alpha$ denote the linear elliptic operator which acts on a test function $\varphi$ by
\begin{equation*}
\Ls_\alpha \varphi : = -\tr\! \left( A_\alpha(y ,\omega) D^2\varphi \right) + b_\alpha\left(y,\omega\right) \cdot D\varphi,
\end{equation*}
and $\Ls = ( \Ls_1,\ldots,\Ls_\m)$ act on $\Phi = (\varphi_1,\ldots,\varphi_\m)$ by
\begin{equation*}
\Ls \Phi : = \left( \Ls_1\varphi_1,\ldots , \Ls_\m \varphi_\m\right).
\end{equation*}
The operator corresponding to the microscopic scale of order $\ep$ is denoted by
\begin{equation} \label{scaling}
\left(\Ls^\ep \Phi\right)\!(x) : = \left( \Ls \Psi\right)\!\left(\frac x\ep \right),
\end{equation}
where $\Psi(x) : = \Phi(\ep x)$. Hence we may write $\Ls^\ep = \left( \Ls^\ep_1,\ldots, \Ls^\ep_\m\right)$ where 
\begin{equation*}
\Ls^\ep_\alpha \varphi = -\ep^2\tr\! \left( A_\alpha\left(\frac x\ep ,\omega \right) D^2\varphi \right) + \ep b_\alpha\left(\frac x\ep,\omega\right) \cdot D\varphi.
\end{equation*}
In view of the above, the eigenvalue problem \eqref{introeq} is written concisely as
\begin{equation} \label{eigconc}
\left\{ \begin{aligned}
& \Ls_\ep \Phi^\ep + Q^\ep \Phi^\ep = \lambda_\ep \Sigma^\ep \Phi^\ep & \mbox{in} & \ U, \\
 & \Phi^\ep = 0 & \mbox{on} & \ \partial U,
\end{aligned} \right.
\end{equation}
where $Q_\ep=Q_\ep(x,\omega)$ denotes the matrix with entries $c_{\alpha\beta}(\frac x\ep,\omega)$ and $\Sigma_\ep$ the matrix with entries $\sigma_{\alpha\beta}(\frac x\ep, \omega)$. For future reference, we also write $Q =Q_1$ and $\Sigma = \Sigma_1$.

\subsection{Preliminary facts concerning principle eigenvalues}
The full coupling assumption \eqref{coupled} endows  the linear operators $\Ls$ and $\Ls_\ep$ with certain positivity properties related to the maximum principle (c.f. Sweers \cite{Sw}). The Krein-Rutman theorem may therefore be invoked to yield the existence of a principal eigenvalue of $\lambda_\ep > 0$ of \eqref{eigconc}, which is simple (has a one-dimensional eigenspace) and corresponds to an eigenfunctions $\Phi^\ep$ with entries $\varphi^\ep_\alpha$ which can be chosen to be positive in $U$. We summarize these facts in the following proposition, a proof of which can be found for example in \cite{MS}.

\begin{prop} \label{eigenvalue}
Under assumptions \eqref{AC11}, \eqref{bounds}, \eqref{ellip}, \eqref{diagdom}, \eqref{coupled} and \eqref{sigma}, the system
\begin{equation} \label{eig}
\left\{ \begin{aligned}
& \Ls \Phi + Q \Phi = \lambda_1 \Sigma \Phi & \mbox{in} & \ U, \\
 & \Phi = 0 & \mbox{on} & \ \partial U,
\end{aligned} \right.
\end{equation}
has a unique eigenvalue $\lambda_1 = \lambda_1(U,\omega) > 0$ corresponding to an eigenfunction $\Phi=\Phi(x,\omega)$ with positive entries. The eigenvalue $\lambda_1$ is simple, i.e., the eigenfunction $\Phi$ is unique up to multiplication by a nonzero constant. 
\end{prop}

It is well known (see \cite{MS}) that the principal eigenvalue $\lambda_1$ is characterized by the variational formula
\begin{multline} \label{max-min}
\lambda_1(U,\omega) = \sup\left\{ \lambda \in \R: \ \mbox{there exists} \ \Psi \in C^2(U)^k \ \mbox{with positive entries} \right. \\
\left. \phantom{\Psi^K}\mbox{such that} \ \Ls\Psi + Q \Psi \geq \lambda \Sigma \Psi \ \mbox{in} \ U \right\},
\end{multline}
where the differential inequality is meant to hold entry-by-entry in the classical sense. It then follows that the eigenvalue $\lambda_1$ is monotone with respect to the domain, i.e., for each $\omega\in \Omega$,
\begin{equation} \label{domain-mono}
\lambda_1(U,\omega) \leq \lambda_1(V,\omega) \quad \mbox{provided that} \ V \subseteq U.
\end{equation}
Recalling the scaling relation \eqref{scaling} between $\Ls$ and $\Ls^\ep$, we see that the existence and properties of the principal eigenvalue $\lambda_\ep$ for the problem \eqref{eigconc} follow from Proposition~\ref{eigenvalue} and, in fact,
\begin{equation} \label{micro-macro}
\ep^2 \lambda_\ep(U,\omega) = \lambda_1\!\left(\ep^{-1} U, \omega\right).
\end{equation}
Notice from \eqref{domain-mono} and \eqref{micro-macro} that, if $V$ is any domain which is star-shaped with respect to the origin (a property which implies that $sV \subseteq tV$ if $0< s\leq t$), then
\begin{equation} \label{eig-mono}
\mbox{the map} \quad \ep \mapsto \lambda^\ep(V,\omega) : = \ep^2 \lambda_\ep(V,\omega) \quad \mbox{is increasing.}
\end{equation}
This monotonicity property, combined with the ergodic theorem, implies (see Proposition~\ref{eigslim} below) that $\lambda^\ep$ converges, almost surely, to a deterministic constant $\lambda_0$ which is independent of the domain $U$.

\section{Main results} \label{MR}

In this section we formulate our main results, Theorems~\ref{HJthm}-\ref{uesc} below. To properly motivate them, we recall what is known in the periodic setting. To obtain the asymptotics \eqref{lamexp} and \eqref{phiexp}, Capdeboscq \cite{C1,C2} introduced the \emph{$\theta$-exponential} periodic cell problem
\begin{equation} \label{capdeb-cp}
\left\{ \begin{aligned}
& -\tr\!\left( A_\alpha(y) D^2 \psi^\theta_\alpha\right) + b_\alpha\left(y\right) \cdot D\psi^\theta_\alpha + \sum_{\beta=1}^\m c_{\alpha\beta}(y) \psi_\beta^\theta = \lambda(\theta) \sum_{\beta=1}^\m \sigma_{\alpha\beta} (y) \psi^\theta_\beta  \quad \mbox{in} \ \Rd,\\
& \psi_\alpha^\theta > 0 \quad \mbox{in} \ \Rd, \quad y\mapsto \exp(\theta\cdot y) \psi_\alpha^\theta (y) \quad \mbox{is periodic,}
\end{aligned} \right.
\end{equation}
for a parameter $\theta\in \Rd$. He proved that, as a function of $\theta$, the map $\theta \mapsto \lambda(\theta)$ is strictly concave and $\lambda (\theta) \to -\infty$ as $|\theta| \to \infty$. This implies that $\lambda$ attains its maximum $\overline \lambda$ at a unique $\overline \theta \in \Rd$. Writing 
\begin{equation} \label{factor}
u^\ep_\alpha(x): = \frac{\varphi^\ep_\alpha(x)}{\psi^{\overline\theta}_\alpha(\tfrac x\ep)} \qquad \mbox{and} \qquad \mu_\ep:= \frac{\lambda_\ep - \overline \lambda}{\ep^2},
\end{equation}
it was then observed that \eqref{introeq} can be rewritten in the form
\begin{equation*}
-\divg \!\left( \overline A(\tfrac x\ep) Du^\ep_\alpha \right) = \mu_\ep\overline \sigma(\tfrac x\ep) u^\ep \quad \mbox{in} \ U,
\end{equation*}
where the coefficients $\overline A$ and $\overline \sigma$ are periodic and depend on the solution of the $\overline \theta$-exponential cell problem (and that of its adjoint). The homogenization of the latter is classical, leading to the expansion \eqref{lamexp} and factorization \eqref{phiexp}.

\medskip

In the general stationary ergodic setting, we cannot expect a factorization of the form \eqref{factor} to hold in any suitable sense. This is related to the fact that, in random environments, correctors do not in general exist (see Lions and Souganidis~\cite{LS1}), and so there is no suitable analogue of the functions $\psi^\theta_\alpha$. This is not merely a technical problem, and goes to the heart of difficult issues in the random setting typically referred to as ``a lack of compactness."

\medskip

Our analysis in the random case follows a different approach. We begin by introducing the classical Hopf-Cole change of variables
\begin{equation} \label{hopf-cole}
\psi^\ep_\alpha(x,\omega) : = -\ep \log \varphi^\ep_\alpha(x,\omega),
\end{equation}
which transforms \eqref{introeq} into the nonlinear system
\begin{multline} \label{HJeq}
-\ep \tr \left( A_\alpha\left(\tfrac x\ep ,\omega \right) D^2\psi_\alpha^\ep \right) + A_\alpha \left( \tfrac x\ep, \omega\right) D\psi_\alpha^\ep \cdot D\psi_\alpha^\ep + b_\alpha\left(\tfrac x\ep,\omega\right) \cdot D\psi_\alpha^\ep \\
- \sum_{\beta=1}^\m \left(c_{\alpha\beta}\left(\tfrac x\ep, \omega\right) - \lambda_\ep(\omega) \sigma_{\alpha\beta} \left(\tfrac x\ep, \omega\right) \right) \exp\left( \ep^{-1}(\psi_\alpha^\ep - \psi_\beta^\ep)\right)  =0  \quad \mbox{in} \ U.
\end{multline}
The study of the behavior of \eqref{HJeq} as $\ep \to 0$ falls within the general framework of random homogenization of viscous Hamilton-Jacobi equations. The latter has been studied in the scalar case by Lions and Souganidis \cite{LS2,LS3}, Kosygina, Rezakhanlou and Varadhan \cite{KRV}, and recently by the authors \cite{AS}. By adapting the methods of \cite{LS3} and \cite{AS}, we prove a homogenization result for \eqref{HJeq}. This allows us to understand some aspects of the behavior as $\ep \to 0$ of \eqref{introeq} and to prove the main result on the asymptotics for $\lambda_\ep$ and $\varphi^\alpha_\ep$, which is Theorem~\ref{MAIN} below.

\medskip

To explain how the effective Hamiltonian arises, we temporarily ``forget" that the eigenvalue $\lambda_\ep(\omega)$ is an unknown in \eqref{HJeq}. This will be accounted for later with the help of Proposition~\ref{eigslim}, below. Therefore we consider the system
\begin{equation} \label{HJeq2}
-\ep \tr \left( A_\alpha\left(\tfrac x\ep ,\omega \right) D^2u_\alpha^\ep \right) + H_\alpha \!\left(D u^\ep_\alpha, \tfrac x\ep, \omega\right) + f_\alpha \!\left( \frac{u_1^\ep}{\ep}, \ldots, \frac{u_\m^\ep}{\ep}, \mu_\ep, \tfrac x\ep, \omega \right) = g \quad \mbox{in} \ U,
\end{equation}
where $\mu_\ep =\mu_\ep(\omega) \geq 0$ is a (possibly random) parameter, $g\in C(U)$ is given, and we define
\begin{equation} \label{Hafa}
\left\{ \begin{aligned}
& H_\alpha(p,y,\omega) : = A_\alpha\! \left( y, \omega\right) p \cdot p+ b_\alpha\!\left(y,\omega\right) \cdot p, \\
& f_\alpha(z_1,\ldots,z_\m, \mu, y, \omega) : = \sum_{\beta=1}^\m \left( \mu \sigma_{\alpha\beta}(y,\omega) - c_{\alpha\beta} (y,\omega) \right) \exp(z_\alpha-z_\beta).
\end{aligned} \right.
\end{equation}
In writing \eqref{HJeq2} we have essentially put \eqref{HJeq} into a more convenient form, replaced $\lambda_\ep$ with $\mu_\ep$, and introduced a function $g$ on the right side. 

Observe that $H_\alpha=H_\alpha(p,y,\omega)$ is convex as well as coercive (it grows quadratically) in $p$, while for all $\xi\in \R$,
\begin{equation} \label{fatrans}
f_\alpha(z_1,\ldots,z_m,\mu,y,\omega) = f_\alpha(z_1+\xi,\ldots,z_m+\xi,\mu,y,\omega).
\end{equation}
In addition, there exists $C> 0$ depending only on the constant in \eqref{bounds} such that, for each $\alpha =1,\ldots, \m$, $z_1,\ldots,z_\m \in \R$, $y\in \Rd$, $\omega\in \Omega$ and for all $\mu_1,\mu_2\ \geq 0$ and $p_1,p_2\in\Rd$,
\begin{multline} \label{finmu}
\big| f_\alpha(z_1,\ldots,z_m,\mu_1,y,\omega) - f_\alpha(z_1,\ldots,z_m,\mu_2,y,\omega) \big| \\ \leq C \Big( \max_{\beta\in \indx} \exp(z_\alpha-z_\beta) \Big) \left|\mu_1-\mu_2\right|.
\end{multline}
and
\begin{equation} \label{Hinp}
\big| H_\alpha(p_1,y,\omega) - H_\alpha(p_2,y,\omega) \big| \leq C \big( 1 + |p_1| + |p_2| \big) |p_1-p_2|.
\end{equation}

\medskip

Our first result is a homogenization assertion for the system \eqref{HJeq2}. We stress that the each of assumptions stated in Section~\ref{prelim} is in force in Theorems~\ref{HJthm}, \ref{MAIN} and~\ref{uesc}.

\begin{thm} \label{HJthm}
Let $U\subseteq \Rd$ be any domain, $\mu_\ep = \mu_\ep(\omega)$ a bounded nonnegative random variable, and assume that for each $\ep > 0$, $\omega\in \Omega$ and $\alpha=1\ldots,\m$, $u^\ep_\alpha=u^\ep_\alpha(\cdot,\omega)$ is a solution of \eqref{HJeq2}. Assume also that there exists $u\in C(U)$ and $\mu \geq 0$ such that, for every $\alpha=1\ldots,\m$, as $\ep \to 0$ and a.s. in $\omega$, $\mu_\ep \rightarrow \mu$ and $u^\ep_\alpha \rightarrow u$ locally uniformly in $U$. Then  $u$ is a solution of the scalar equation
\begin{equation} \label{HJeff}
\overline H(Du, \mu) = g \quad \mbox{in} \ U, 
\end{equation}
with the effective Hamiltonian $\overline H:\Rd \times \R_+ \to \R$ characterized in Proposition~\ref{mainstep} below.
\end{thm}

We will see in Proposition~\ref{HAM} that $p \mapsto \overline H(p,\mu)$ is convex and coercive for each $\mu \geq 0$, while $\mu \mapsto \overline H (p,\mu)$ is strictly increasing and $\overline H(0,0) \leq 0$. Therefore, we may define $\overline \lambda$ to be the largest value of $\mu$ for which the graph of $p\mapsto \overline H(\cdot,\mu)$ touches zero, i.e.,
\begin{equation} \label{lbar}
\overline \lambda : = \sup\left\{ \mu \geq 0 : \min_{p\in  \Rd} \overline H (p, \mu) \leq 0 \right\}.
\end{equation}
Since $\overline H$ is continuous, 
\begin{equation}\label{minzeroa}
0 = \min_{p\in \Rd} H(p,\overline \lambda). 
\end{equation}
We know that, in the random environment, $\overline H$ may have a ``flat spot" at its minimum. That is, the set
\begin{equation}\label{barsarezero}
\Theta := \argmin \overline H(\cdot,\overline\lambda) = \{ p \in \Rd \, : \, \overline H(p,\overline \lambda) = 0 \}
\end{equation}
may have a nonempty interior. However, in certain cases, for example, if $p\mapsto \overline H(p,\mu)$ is strictly convex, then  $p\mapsto \overline H(p, \overline \lambda)$ necessarily has a unique minimum, that is, $\Theta = \{ \overline \theta\}$. In the latter situation, we obtain that the eigenfunctions exhibit concentration behavior in the sense of \eqref{concent} below.

We emphasize that $\overline\lambda$ and $\overline\theta$ are deterministic quantities which are independent of the domain $U$.

\medskip

We now state the result regarding the asymptotics of the linear system \eqref{introeq}. In what follows, $U$ is taken to be a smooth, bounded domain and $\lambda_\ep$ and $\varphi^\ep_\alpha$ together solve the system \eqref{introeq}-\eqref{introbc}, subject to the normalization 
\begin{equation}\label{norlizz}
\varphi_1^\ep(x_0) = 1 \quad \mbox{for some distinguished} \ x_0 \in U \ \mbox{and a.s. in} \ \omega.
\end{equation}
The following result characterizes the limit of the eigenvalues $\lambda_\ep$, and uncovers the concentration behavior of the eigenfunctions $\varphi^\ep_\alpha$ in the case that $\overline H\big(\cdot,\overline \lambda \big)$ achieves its minimum at a unique point $\overline \theta$. 

\begin{thm} \label{MAIN}
The principle eigenvalue $\lambda_\ep(\omega,U)$ of  \eqref{introeq}-\eqref{introbc} satisfies
\begin{equation} \label{drift}
\ep^2 \lambda_\ep(\omega,U) \rightarrow \overline\lambda \quad\mbox{as} \ \ep \to 0 \quad \mbox{and a.s. in} \ \omega,
\end{equation}
where $\overline \lambda$ is given by \eqref{lbar}. Suppose in addition that $\overline H(\cdot,\overline \lambda)$ attains its minimum at a unique point $\overline \theta\in \Rd$. Then, for each $\alpha=1,\ldots,\m$ and as $\ep \to 0$, 
\begin{equation} \label{concent}
-\ep \log \varphi^\ep_\alpha(x,\omega) \rightarrow \overline\theta\cdot (x-x_0) \quad \mbox{locally uniformly in} \ U \ \mbox{and a.s. in} \ \omega.
\end{equation}
\end{thm}

We present a sufficient condition for the strict convexity of $\overline H$. The following theorem states that the effective Hamiltonian is strictly convex in $p$ under the additional assumption that the random environment is \emph{uniquely ergodic}. Roughly speaking, this means that the limit \eqref{erglim} in the ergodic theorem is uniformly with respect to the translations. The precise definition follows.

\begin{definition} \label{uedef}
The action of the group $( \tau_y)_{y\in \Rd}$ on the environment $(\Omega, \mathcal F,\Prob)$ is \emph{uniquely ergodic} if, for every $\mathcal F$-measurable $f:\Omega \to \R$ such that $\E |f| < \infty$, there exists a subset $\widetilde\Omega \subseteq \Omega$ of full probability such that, for every $\omega\in \widetilde \Omega$,
\begin{equation} \label{uecond}
\lim_{R\to \infty} \sup_{z\in \Rd} \ \bigg| \fint_{B(z,R)} f(\tau_y \omega) \, dy - \E f \bigg| = 0.
\end{equation}
\end{definition}

Equivalently, the action of $( \tau_y)_{y\in \Rd}$ on the environment is uniquely ergodic if and only if $\Prob$ is the unique $\mathcal F$-measurable probability measure which is invariant under the action.

The space of almost periodic functions may be embedded into the stationary ergodic setting (c.f. \cite{AF}), and it is easy to see that the resulting random environment must be uniquely ergodic. Therefore \eqref{concent} holds in particular in the almost periodic setting. The inclusions are proper: it is well-known that there exist stationary ergodic environments which are not uniquely ergodic (for example, any iid environment cannot be uniquely ergodic), and uniquely ergodic environments which are not equivalent to translations of an almost periodic function \cite{AW}.

\begin{thm} \label{uesc}
Assume that the action of $( \tau_y)_{\Rd}$ on $(\Omega, \mathcal F,\Prob)$ is uniquely ergodic. Then, for each $\mu \geq 0$, 
\begin{equation}\label{}
p \mapsto \overline H(p,\mu) \quad \mbox{is strictly convex.}
\end{equation}
Hence $\Theta = \{ \overline \theta \}$, for some $\overline \theta \in \Rd$, and the concentration phenomenon \eqref{concent} holds.
\end{thm}

\section{The auxiliary macroscopic problem} \label{H}

For each fixed $\delta > 0$, $\mu \geq 0$ and $p \in  \Rd$, we introduce the \emph{auxiliary macroscopic system}
\begin{equation} \label{HJaux}
\delta v^\delta_\alpha - \tr ( A_\alpha\!\left(y, \omega \right)\! D^2 v^\delta_\alpha ) + H_\alpha( p+ Dv^\delta_\alpha,y,\omega) + f_\alpha\!\left( v^\delta_1, \ldots, v^\delta_\m, \mu, y, \omega\right) = 0 \quad \mbox{in} \ \Rd,
\end{equation}
with $H_\alpha$ and $f_\alpha$ defined in \eqref{Hafa}. In the periodic setting, \eqref{HJaux} is known as the ``cell problem," and it is central to the homogenization of Hamilton-Jacobi equations. In this section we establish the well-posedness of \eqref{HJaux}, a necessary precursor to the next section, where we construct $\overline H$ via a limit procedure using~$v^\delta$. 

\medskip

Following the usual viscosity theoretic approach, we first give a comparison principle for \eqref{HJaux}. The structural assumptions in Section~\ref{prelim} yield the following proposition, which is essentially due to Ishii and Koike \cite{IK}. The result in \cite{IK} was stated only for equations in bounded domains, but we extend it to $\Rd$ via a simple argument using the convexity of $H_\alpha$. 

\begin{prop}\label{vdeltacmp}
Fix $\delta > 0$, $\mu\geq 0$, $p \in \Rd$ and $\omega\in \Omega$. Suppose that $u_\alpha \in \USC(\Rd)$ is bounded above and $v_\alpha \in \LSC(\Rd)$ is bounded below, and $u=u_\alpha$ and $v=v_\alpha$ satisfy, for each $\alpha=1,\ldots,\m$,
\begin{equation*}
\delta u_\alpha - \tr ( A_\alpha\!\left(y, \omega \right)\! D^2 u_\alpha ) + H_\alpha( p+ Du_\alpha,y,\omega) + f_\alpha\!\left( u_1, \ldots, u_\m, \mu, y, \omega\right) \leq 0 \quad \mbox{in} \ \Rd,
\end{equation*}
and
\begin{equation*}
\delta v_\alpha - \tr ( A_\alpha\!\left(y, \omega \right)\! D^2 v_\alpha ) + H_\alpha( p+ Dv_\alpha,y,\omega) + f_\alpha\!\left( v_1, \ldots, v_\m, \mu, y, \omega\right) \geq 0 \quad \mbox{in} \ \Rd.
\end{equation*}
Then, for every $\alpha=1,\ldots,\m$, $u_\alpha\leq v_\alpha$ in $\Rd$.
\end{prop}
\begin{proof}
According to the structural assumptions, we may select $k >0$ sufficiently large (depending on $\delta$) so that, for each $\alpha=1,\ldots,\m$, the function $\varphi(y):= k - (1+|y|^2)^{1/2}$ is a smooth solution of
\begin{equation*}
\delta \varphi - \tr ( A_\alpha\!\left(y, \omega \right)\! D^2 \varphi ) + H_\alpha( p+ D \varphi,y,\omega)  \leq 0 \quad \mbox{in} \ \Rd.
\end{equation*}
Modify $u_\alpha$ by defining, for each $\ep > 0$,
\begin{equation*}
u_{\alpha,\ep}(y):= (1-\ep) u_\alpha(y) + \ep\varphi(y).
\end{equation*}
Formally, using the convexity of $H_\alpha$ and \eqref{fatrans}, for each $\alpha =1,\ldots,\m$, we have
\begin{equation*}
\delta u_{\alpha,\ep} - \tr ( A_\alpha\!\left(y, \omega \right)\! D^2u_{\alpha,\ep}  ) + H_\alpha( p+ Du_{\alpha,\ep} ,y,\omega) + f_\alpha\!\left( u_{1,\ep} , \ldots, u_{m,\ep} , \mu, y, \omega\right) \leq 0 \quad \mbox{in} \ \Rd.
\end{equation*}
This can be made rigorous  either by using the fact that $\varphi$ is smooth or by applying \cite[Lemma A.1]{AS}. Since $v_\alpha$ is bounded below and $u_{\alpha,\ep}(y) \rightarrow -\infty$ as $|y| \to \infty$, for all $R> 0$ sufficiently large and for each $\alpha=1,\ldots,\m$, we have
\begin{equation*}
u_{\alpha,\ep} \leq v_{\alpha} \quad \mbox{in} \ \Rd\setminus B_R.
\end{equation*}
It then follows from  \cite[Theorem 4.7]{IK} that, for each $\alpha=1,\ldots,\m$,
\begin{equation*}
u_{\alpha,\ep} \leq v_{\alpha} \quad \mbox{in} \ \Rd,
\end{equation*}
and, after sending $\ep \to 0$, the conclusion.
\end{proof}

The unique solvability of \eqref{HJaux} follows easily from Proposition~\ref{vdeltacmp} and the Perron method.

\begin{prop} \label{vdelta}
For each $\delta > 0$, $\mu\geq 0$, $p \in \Rd$, and $\omega \in \Omega$, there exists a unique bounded viscosity solution $v^\delta(\cdot, \omega;p,\mu) = (v^\delta_1(\cdot,\omega;p,\mu),\ldots, v^\delta_\m(\cdot,\omega;p,\mu))\in C( \Rd)^\m$ of \eqref{HJaux}, which is stationary. Moreover, there exist  $C,c>0$, depending only on the constants in \eqref{ellip}, \eqref{bounds} and \eqref{sigma}, such that, for each $\omega \in \Omega$ and $\alpha=1,\ldots,\m$,
\begin{equation} \label{deltabnd}
-\big(\Lambda|p|^2 + C(|p| + \mu) \big) \leq \delta v^\delta_\alpha(\cdot, \omega;p,\mu) \leq -\big( \lambda |p|^2 - C(|p| +1) \big) \quad \mbox{in} \ \Rd.
\end{equation}
\end{prop}
\begin{proof}
We suppress dependence on $\omega$, since it plays no role in the proof. Denote
\begin{equation*}
\underline{\varphi}_\alpha(y): =  - \frac1\delta \big( \Lambda|p|^2 + C(|p| +\mu) \big) \qquad \mbox{and} \qquad \overline{\varphi}_\alpha(y) : = - \frac1\delta \big( \lambda|p|^2 - C(|p|+1) \big).
\end{equation*}
It is easy to check, using \eqref{ellip}, \eqref{bounds}, \eqref{diagdom} and \eqref{sigma}, that $\underline{\varphi}_\alpha$ and $\overline{\varphi}_\alpha$ are, respectively, a subsolution and supersolution of \eqref{HJaux} in $\Rd$. Define, for each $\alpha=1,\ldots,\m$,
\begin{multline*}
v^\delta_\alpha(y): = \sup \big\{ \varphi_\alpha (y): \varphi_1,\ldots,\varphi_\m \in \USC(\Rd) \ \mbox{are bounded above and} \\ \varphi=(\varphi_1,\ldots,\varphi_\m) \ \mbox{is a subsolution of} \ \eqref{HJaux} \ \mbox{in} \ \Rd \big\}.
\end{multline*}
It is clear from the definition above and Proposition~\ref{vdeltacmp} that $\underline{\varphi}_\alpha \leq v^\delta_\alpha \leq \overline{\varphi}_\alpha$ in $\Rd$, which gives \eqref{deltabnd}. Standard arguments from the theory of viscosity solutions, utilizing Proposition~\ref{vdeltacmp} imply that $v^\delta\in C(\Rd)^\m$ and $v^\delta$ is a solution of \eqref{HJaux}. We refer to \cite{CIL} and to Section 3 of \cite{IK} for details. According to Proposition~\ref{vdeltacmp}, $v^\delta$ is the unique bounded solution of \eqref{HJaux}. The stationarity of $v^\delta$ is an immediate consequence of the stationarity of the coefficients and the uniqueness of $v^\delta$.
\end{proof}

In the following proposition, we use the Harnack inequality for linear elliptic systems (see Busca and Sirakov~\cite{BS}) to obtain an estimate, independently of $\delta$, on the difference between $v^\delta_\alpha$ and $v^\delta_\beta$. The Bernstein method then yields uniform Lipschitz bounds on $v^\delta$.

\begin{prop} \label{HJauxLip}
For each $\delta > 0$, $\mu \geq 0$, $p\in \Rd$ and $\omega \in \Omega$, the unique bounded solution $v^\delta(\cdot,\omega;p,\mu)$ of \eqref{HJaux} belongs to $C^2(\Rd)^\m$ and there exists $C>0$, which depends on upper bounds for $|p|$ and $\mu$ but is independent of $\delta$ and $\omega$, such that 
\begin{equation}\label{collapse}
\max_{\alpha,\beta\in \indx} \esssup_{\Rd} \big| v^\delta_\alpha(\cdot,\omega) - v^\delta_\beta(\cdot, \omega)\big| \leq C
\end{equation}
and
\begin{equation}\label{lipschitz}
\max_{\alpha\in \indx} \esssup_{\Rd} \big| Dv^\delta_\alpha(\cdot, \omega) \big| \leq C. 
\end{equation}
\end{prop}
\begin{proof}
For convenience, we omit the explicit dependence on $\omega$ since it has no role in the argument. The smoothness of $v^\delta$ is immediate from classical elliptic regularity. The estimate \eqref{collapse} is a consequence of a Harnack inequality for a linear cooperative systems. To see this, we observe that $w^\delta=w^\delta_\alpha$ with $w^\delta_\alpha : = \exp(-v^\delta_\alpha)$ is a classical solution of
\begin{multline*}
-\tr ( A_\alpha(y) D^2w^\delta_\alpha) + \left( 2A_\alpha(y) p + b_\alpha(y) \right) \cdot Dw^\delta_\alpha \\ - (A_\alpha(y) p\cdot p + b_\alpha(y) \cdot p + \delta v^\delta_\alpha) w^\delta_\alpha + \sum_{\beta=1}^\m ( c_{\alpha\beta}(y) - \mu\sigma_{\alpha\beta}) w^\delta_\beta = 0 \quad \mbox{in} \ \Rd,
\end{multline*}
which is a cooperative, fully coupled linear system, and thanks to \eqref{deltabnd}, has bounded coefficients. The Harnack inequality found in \cite[Corollary 8.1]{BS} yields that\begin{equation*}
w^\delta_\alpha(y) \leq C_0 w^\delta_\beta(y) \quad \mbox{for each} \ y\in  \Rd \ \mbox{and} \ \alpha,\beta=1,\ldots,\m.
\end{equation*}
Rewriting this inequality in terms of $v^\delta_\alpha$ and $v^\delta_\beta$ yields \eqref{collapse}.

\medskip

According to \eqref{collapse}, the last term $f_\alpha(v^\delta_1,\ldots,v^\delta_\m,\mu,y)$ in \eqref{HJaux} is bounded independently of $\delta$. We next use the Bernstein method to obtain the Lipschitz estimate \eqref{lipschitz}. Although it proceeds very similarly as the proof of \cite[Proposition 6.11]{LS2} for the case of a scalar equation, for the convenience of the reader we give complete details here because the form of the system \eqref{HJaux} complicates the argument somewhat. For ease of notation we do not display the explicit dependence of $v^\delta_\alpha$ on $\delta$. Select a cutoff function $\varphi \in C^\infty( \Rd)$ such that
\begin{equation}
0 \leq \varphi \leq 1, \ \ \varphi \equiv 1 \ \mbox{on} \ B_{1}, \ \ \varphi \equiv 0 \ \mbox{in} \  \Rd \!\setminus \! B_2, \ \ \left| D^2 \varphi \right| \leq C\varphi^{\frac12} \ \ \mbox{and} \ \ |D\varphi| \leq C\varphi^{\frac34}. \label{varphiest}
\end{equation}
It suffices to choose for example $\varphi = \psi^4$, where $\psi$ is a cutoff function satisfying the first three conditions of \eqref{varphiest}. Denote $\xi_\alpha :=  |Dv_\alpha|^2$ and $w_\alpha : = \varphi \xi_\alpha = \varphi |Dv_\alpha|^2$. An easy computation yields
\begin{align}
Dv_\alpha & = \xi_\alpha D\varphi + 2\varphi D^2v_\alpha Dv_\alpha, \label{easyderv} \\
\label{easyderv2}
D^2w_\alpha & = \xi_\alpha D^2\varphi + 2 D\varphi  \otimes (D^2v_\alpha Dv_\alpha) + \varphi (D^3v_\alpha Dv_\alpha + D^2v_\alpha D^2v_\alpha).
\end{align}
Differentiating \eqref{HJaux} with respect to $y_i$, multiplying the result by $\varphi v_{\alpha,y_i}$ and using \eqref{easyderv} and \eqref{easyderv2}, we obtain after some calculation that, on the support of $\varphi$,
\begin{multline} \label{bernmess}
\delta w_\alpha -  \tr\!\left( A_\alpha D^2 w_\alpha \right) + \varphi \tr(D^2v_\alpha A_\alpha D^2v_\alpha)  - \varphi Dv_\alpha \cdot \tr(D_yA_\alpha D^2v_\alpha) \\
+ \xi_\alpha \tr (A_\alpha D^2\varphi) + A_\alpha \varphi^{-1}D\varphi\cdot (Dw_\alpha -\xi_\alpha D\varphi)\\
+ \varphi Dv_\alpha \cdot \left( (D_y A_\alpha (p+Dv_\alpha) + D_y b_\alpha ) \cdot (p+Dv_\alpha) \right) + A_\alpha(Dw_\alpha - \xi_\alpha D\varphi) \cdot (p+Dv_\alpha) \\
+\tfrac 12  b_\alpha \cdot (Dw_\alpha - \xi_\alpha D\varphi) + \varphi Dv_\alpha \cdot \sum_{\beta=1}^\m e^{v_\alpha-v_\beta} D_y\!\left( \mu \sigma_{\alpha\beta} - c_{\alpha\beta} \right) \\
+ \varphi \sum_{\beta=1}^\m e^{v_\alpha-v_\beta} \left( \mu \sigma_{\alpha\beta} - c_{\alpha\beta} \right) \left( |Dv_\alpha|^2 - Dv_\alpha\cdot Dv_\beta \right) = 0.
\end{multline}
Now suppose that $\widehat x \in B_2$ and $\alpha' =1,\ldots,\m$ are such that 
\begin{equation} \label{maxmax}
w_{\alpha'}(\widehat x) = \max_{\alpha\in \indx} \sup_{x\in \Rd} w_\alpha(x).
\end{equation}
To simplify the notation we assume that $\alpha' =1$. Then $Dw_1(\widehat x) = 0$ and $D^2w_1(\widehat x) \leq 0$. Using these together with \eqref{collapse}, \eqref{varphiest} and the observation that \eqref{maxmax} implies that, at $x=\widehat x$, 
\begin{equation*}
\sum_{\beta=1}^\m e^{v_1-v_\beta} \left( \mu \sigma_{1\beta} - c_{1\beta} \right) \left( |Dv_1|^2 - Dv_1\cdot Dv_\beta \right) \geq 0,
\end{equation*}
after some work we obtain, from \eqref{bernmess},
\begin{equation*} \label{bernmess2}
\varphi \tr (AM^2) \leq C \!\left( \varphi |q|( |M|+1) +|q|^2 (|D^2\varphi| +  |D\varphi| + \varphi^{-1}|D\varphi|^2) + |q|^3( \varphi + |D\varphi|)  \right),
\end{equation*}
where for convenience we have written $M : = D^2v_1(\widehat x)$, $q : = Dv_1(\widehat x)$ and $A=A_1(\widehat x)$. Applying some elementary inequalities and using \eqref{Alip} we get
\begin{equation*}
\varphi\tr (AM^2) \leq C_\eta\! \left( \varphi + |q|^2 \varphi^{\frac12} + |q|^3 \varphi^{\frac34} \right) + \eta \varphi |M|^2,
\end{equation*}
where $\eta> 0$ is selected below. By \eqref{ellip} and the Cauchy-Schwarz inequality,
\begin{equation*}
| M|^2 \leq C (\tr (AM))^2  \leq C \tr (AM^2).
\end{equation*}
Hence by making $\eta > 0$ small we get 
\begin{equation} \label{trAM2}
\varphi\tr(AM^2) \leq C\! \left( 1 + |q|^2 \varphi^{\frac12} + |q|^3 \varphi^{\frac34} \right).
\end{equation}
Using the PDE \eqref{HJaux} and the estimates \eqref{deltabnd}, \eqref{collapse}, we have
\begin{equation} \label{trAM3}
(\tr(AM))^2 = \left( H_1(p+Dv_1,y) + f_1(v_1,\ldots,v_k,\mu,y) + \delta v_1 \right)^2 \geq \lambda(|q|^2 - C)^2.
\end{equation}
Putting \eqref{trAM2} and \eqref{trAM3} together, we obtain that
\begin{equation*}
|q|^4 \varphi \leq C\!\left( 1 + |q|^2 \varphi^{\frac12} + |q|^3 \varphi^{\frac34} \right).
\end{equation*}
This yields an upper bound on $|q|^4\varphi$ and hence
\begin{equation*}
(w_1(\widehat x))^2 = |q|^4 \varphi^2 \leq |q|^4 \varphi \leq C. 
\end{equation*}
Thus
\begin{equation*}
\max_{\alpha\in \indx} \sup_{B_1} \big| Dv_\alpha \big|^2 \leq \max_{\alpha\in \indx} \sup_{\Rd} \big| w_\alpha \big| = w_1(\widehat x) \leq C.
\end{equation*}
Since the constant $C> 0$ in the last inequality is independent of the fact we centered our ball at the origin, the proof is complete.
\end{proof}

\begin{remark}
The essential boundedness and stationarity of $v_\alpha^\delta$ and \eqref{lipschitz} imply that
\begin{equation}\label{EDv0}
\E\big[Dv_\alpha^\delta(0,\cdot)\big] = 0.
\end{equation}
This follows from an argument of Kozlov~\cite{K}, see also \cite[Lemma A.5]{AS}.
\end{remark}

We conclude this section by studying the dependence of the solution of \eqref{HJaux} on the parameters $p$ and $\mu$. It is a necessary ingredient in the proof of Proposition~\ref{mainstep} and  will yield important properties of $\overline H$.

\begin{prop} \label{CDE}
Let $v^\delta=v^\delta_\alpha(\cdot,\omega;p,\mu)$ be as in Proposition~\ref{vdelta}. Then: (i) for each $k> 0$ there exists $C > 0$, depending on $\m$, $k$ and the constant in \eqref{bounds}, such that, for all $\delta > 0$, $\omega\in \Omega$, $p_1,p_2\in \Rd$ and $0 \leq \mu \leq k$,
\begin{equation} \label{CDEinp}
\max_{\alpha\in\indx}\sup_{\Rd} \, \delta \big| v^\delta_\alpha (\cdot,\omega;p_1,\mu) - v^\delta_\alpha(\cdot,\omega;p_2,\mu) \big| \leq C\big( 1+\big|p_1\big| + \big|p_2\big| \big)\, \big|p_1-p_2\big|.
\end{equation}
(ii) for each $k> 0$ there exist $C,c> 0$, depending only on $\m$ and the constants in \eqref{sigma} and \eqref{bounds},  such that, for all $\delta > 0$, $\omega\in \Omega$, $p\in B(0,k)$, $0 \leq \mu_1 \leq \mu_2 \leq k$ and $\alpha =1,\ldots,\m$,
\begin{equation} \label{CDEinmu}
c(\mu_2-\mu_1) \leq \delta v_\alpha^\delta(\cdot,\omega;p,\mu_1) - \delta v_\alpha^\delta(\cdot,\omega;p,\mu_2) \leq C (\mu_2-\mu_1) \quad \mbox{in} \ \Rd.
\end{equation}
\end{prop}
\begin{proof}
The proof of \eqref{CDEinp} closely follows the proof of \cite[Lemma 4.6]{AS}. We fix $p_1,p_2\in \Rd$, $\mu\geq 0$ and $\omega\in\Omega$ and write $v^\delta_{\alpha,i}(y):= v^\delta_\alpha(y,\omega;p_i,\mu)$ for $i\in \{1,2\}$ and $\alpha=1,\ldots,\m$. Define $\lambda:= (1+|p_1|+|p_2|)^{-1}|p_1-p_2|$ and set
\begin{equation*}
w^\delta_\alpha(y) :  = (1-\lambda) v^\delta_{\alpha,2}(y) = (1-\lambda) \big( (p_2-p_1)\cdot y +v^\delta_{\alpha,2}(y) \big) + \lambda \big( \lambda^{-1} (1-\lambda)(p_1-p_2)\cdot y\big).
\end{equation*}
It is easy to check, using the convexity of $H_\alpha$ and \eqref{fatrans}, that $w^\delta_\alpha$ satisfies
\begin{multline*}
\delta w^\delta_\alpha - \tr \!\big( A_\alpha(y, \omega ) D^2 w^\delta_\alpha \big) + H_\alpha( p+ Dw^\delta_\alpha,y,\omega) + f_\alpha\!\left( w^\delta_1, \ldots, w^\delta_\m, \mu, y, \omega\right) \\ \leq \lambda H_\alpha\big( \lambda^{-1} (p_1-(1-\lambda)p_2),y,\omega\big) \quad \mbox{in} \ \Rd.
\end{multline*}
Since
\begin{equation*}
\lambda^{-1} \big| p_1 - (1-\lambda) p_2 \big| \leq 1 + |p_1 + 2|p_2|, 
\end{equation*}
there exists a constant $C>0$ depending only on the constant in \eqref{bounds}, so that
\begin{equation*}
\lambda H_\alpha\big( \lambda^{-1} (p_1-(1-\lambda)p_2),y,\omega\big) \leq C\big(1+|p_1|+|p_2| \big) \big| p_1-p_2 \big|.
\end{equation*}
By subtracting $\delta^{-1}C\big(1+|p_1|+|p_2| \big) \big| p_1-p_2 \big|$ from $w^\delta_{\alpha}$ we obtain a subsolution of \eqref{HJaux}, and Proposition~\ref{vdeltacmp} yields
\begin{equation*}
(1-\lambda) v^\delta_{\alpha,2} = w^\delta_\alpha \leq v^\delta_{\alpha,1} + \delta^{-1}C\big(1+|p_1|+|p_2| \big) \big| p_1-p_2 \big|.
\end{equation*}
Using \eqref{deltabnd} and rearranging, we obtain
\begin{equation*}
v^\delta_{\alpha,2} - v^\delta_{\alpha,1} \leq \delta^{-1} C \big( (1+|p_1|+|p_2| \big) \big| p_1-p_2 \big|
\end{equation*}
where $C$ depends additionally on an upper bound for $\mu$. Multiplying by $\delta$ and repeating the argument with the indices reversed yields \eqref{CDEinp}.

The first inequality in \eqref{CDEinmu} easily follows from \eqref{collapse}, \eqref{sigma}, and Proposition~\ref{vdeltacmp}. Indeed, the function $v^\delta(y,\omega;p,\mu_1) - \delta^{-1} c (\mu_2-\mu_1)$ is a supersolution of \eqref{HJaux} for $\mu = \mu_2$, for a $c>0$ with appropriate dependencies. The second inequality follows similarly, since the function $v^\delta(y,\omega;p,\mu_1) - \delta^{-1} C (\mu_2-\mu_1)$ is a subsolution the same equation for sufficiently large $C> 0$.
\end{proof}

\section{Construction of the effective Hamiltonian $\overline H(p,\mu)$} \label{Hbar}

In the next proposition, which is the analogue of \cite[Proposition 5.1]{AS}, we identify $\overline H$ as a limit of $\delta v^\delta$ as $\delta \to 0$, and construct a subcorrector. The argument is based on ideas introduced in \cite{LS3}.

\begin{prop} \label{mainstep}
There exist a continuous function $\overline H:\Rd\times\R_+ \to \R$ and a subset $\Omega_1\subseteq \Omega$ of full probability, such that, for every $p\in \Rd$, $\mu \geq 0$, $R> 0$, and $\alpha=1,\ldots,\m$,
\begin{equation} \label{mainstepeq}
\lim_{\delta \to 0}\sup_{y\in B_{R/\delta}} \left| \delta v^\delta_\alpha (y,\cdot;p,\mu) + \overline H(p,\mu) \right| = 0 \quad \mbox{in} \ L^1(\Omega, \Prob),
\end{equation}
and, for every $\omega\in \Omega_1$,
\begin{equation}\label{liminfH}
\overline H(p,\mu) = - \liminf_{\delta \to 0} \delta v^\delta(0,\omega;p,\mu) \quad \mbox{a.s. in} \ \omega.
\end{equation}
Moreover, there exists $w = w_\alpha(y,\omega;p,\mu)$ such that, for every $(p,\mu,\alpha) \in \Rd\times\R \times \indx$ and $\omega \in \Omega_1$, $w_\alpha(\cdot,\omega;p,\mu)\in C^{0,1}(\Rd)$ as well as
\begin{equation}\label{}
Dw_\alpha(\cdot,\cdot;p,\mu) \quad \mbox{is stationary and} \ \ \E\big[ Dw_\alpha(0,\cdot;p,\mu) \big] =0,
\end{equation}
\begin{equation} \label{sublininfty}
\lim_{|y|\to\infty} |y|^{-1}w_\alpha(y,\omega) = 0
\end{equation}
and
\begin{equation}\label{HJauxlim}
- \tr ( A_\alpha\!\left(y, \omega \right)\! D^2 w_\alpha ) + H_\alpha( p+ Dw_\alpha,y,\omega) + f_\alpha\!\left( w_1, \ldots,w_\m, \mu, y, \omega\right) \leq \overline H(p,\mu) \quad \mbox{in} \ \Rd.
\end{equation} 
\end{prop}
\begin{proof}
Proposition~\ref{CDE} allows us to prove the claim for fixed $p \in \Rd$ and $\mu \geq 0$, and then to intersect relevant subsets of $\Omega$ for a countable dense subset of $(p,\mu)$ in $\Rd \times \R_+$. We therefore fix $p$ and $\mu$ and omit the dependence on these variables for ease of notation. 

For $\alpha =1,\ldots,\m$, define
\begin{equation}\label{defvhat}
\widehat v^\delta_\alpha(y,\omega) := v^\delta_\alpha(y,\omega) - v^\delta_1(0,\omega).
\end{equation}
The estimates \eqref{deltabnd}, \eqref{collapse} and \eqref{lipschitz} together with the stationarity of $v^\delta$ are sufficient for the extraction of a subsequence $\delta_j \to 0$ such that, for every $R> 0$, as $j \to \infty$,
\begin{equation*}
\left\{ \begin{aligned}
& -\delta_j v^{\delta_j}_\alpha \rightharpoonup \overline H \quad \mbox{weakly-}\!\ast \ \mbox{in} \ L^\infty(B_R\times \Omega), \\
& \widehat v^{\delta_j}_\alpha \rightharpoonup w_\alpha \quad \mbox{weakly-}\!\ast \ \mbox{in} \ L^\infty(B_R\times\Omega), \\
& D\widehat v^{\delta_j}_\alpha \rightharpoonup Dw_\alpha \quad \mbox{weakly-}\!\ast \ \mbox{in}  \ L^\infty(B_R\times\Omega),
\end{aligned} \right.
\end{equation*}
for a deterministic constant $\overline H = \overline H(p,\mu)$ and functions $w_\alpha(\cdot,\omega) \in C^{0,1}(\Rd)$. Standard arguments from the theory of viscosity solutions (using in particular the convex structure of \eqref{HJaux} and the equivalence of distributional and viscosity solutions for linear inequalities, c.f. Ishii~\cite{I}) yield that, for each $\alpha=1,\ldots,\m$, $w=w_\alpha$ is a solution, a.s. in $\omega$, of the system~\eqref{HJauxlim}. We emphasize that in deriving \eqref{HJauxlim} we rely crucially on the convexity of $H_\alpha(p,y,\omega)$ in $p$ and of $f_\alpha(z_1,\ldots,z_k,\mu,y,\omega)$ in the differences $z_\alpha-z_\beta$. According to \eqref{lipschitz}, the gradients $Dw_\alpha$  satisfy
\begin{equation}\label{walphlip}
\sup_{\alpha\in \indx} \esssup_{\Rd\times \Omega} \big| Dw_\alpha \big| \leq C,
\end{equation}
and they inherit the stationarity property from the sequence $v^{\delta_j}_\alpha$. From \eqref{EDv0} we deduce that
\begin{equation*}
\E\big[Dw_\alpha\big] = \lim_{j \to \infty} \E\big[D\widehat v^{\delta_j}_\alpha\big] = \lim_{j \to \infty} \E\big[D v^{\delta_j}_\alpha\big] = 0.
\end{equation*}
The ergodic theorem (c.f. Kozlov \cite{K} or \cite[Lemma A.5]{AS}) yields \eqref{sublininfty}.

Having identified $\overline H$, we must show that it characterizes the full limit of $-\delta v^\delta_\alpha$. The key step is to use the comparison principle to show that
\begin{equation} \label{cmpineq}
-\overline H \leq \liminf_{\delta \to 0} \delta v^\delta_\alpha(0,\omega) \quad  \mbox{and a.s. in} \ \omega.
\end{equation}
Denote by $\widetilde \Omega$ the subset of $\Omega$ of full probability consisting $\omega$ for which \eqref{sublininfty} and \eqref{HJauxlim} hold, as well as 
\begin{equation}\label{walphlipy}
\sup_{\Rd} |Dw_\alpha(\cdot,\omega)|\leq C,
\end{equation}
the latter condition holding on a subset of full probability by Fubini's theorem and \eqref{walphlip}.

Fix $\omega\in \widetilde \Omega$. Choose $\delta,\eta> 0$, let $\gamma > 0$ be a constant to be selected below, and define, for each $\alpha=1,\ldots,\m$,
\begin{equation*}
\widehat w^\delta_\alpha(y,\omega): = w_\alpha(y,\omega) - \left(\overline H+\eta\right)\! \delta^{-1} - \gamma(1+|y|^2)^{1/2}.
\end{equation*}
Due to \eqref{walphlipy} and the fact that $|D\widehat w^\delta_\alpha| \leq |Dw_\alpha| + C\gamma$, we see that $\widehat w^\delta=(\widehat w^\delta_1,\ldots,\widehat w^\delta_\m)$ is a solution of the system of inequalities
\begin{multline} \label{HJauxlimW}
\delta \widehat w^\delta_\alpha- \tr ( A_\alpha\!\left(y, \omega \right)\! D^2 \widehat w^\delta_\alpha ) + H_\alpha( p+ D\widehat w^\delta_\alpha,y,\omega) + f_\alpha(\widehat w^\delta_1,\ldots, \widehat w^\delta_\m,\mu,y,\omega)\\ \leq \delta w_\alpha - \eta + C\gamma \quad \mbox{in} \ \Rd\quad (\alpha=1,\ldots,\m).
\end{multline}
Fix a constant $r> 0$ to be selected below. Choosing $\gamma := \eta/(2C)$ and applying \eqref{sublininfty}, we may estimate the right side of \eqref{HJauxlimW} for $|y| \leq r$ by
\begin{equation} \label{rsubball}
\delta w_\alpha - \eta + C\gamma \leq \delta C_\eta + \delta\eta^3 r - \frac{1}{2}\eta. 
\end{equation}
Next we observe that by \eqref{deltabnd}, the definition of $\widehat w^\delta_\alpha$ and our choice of $\gamma$, we have
\begin{equation} \label{rsubbdry}
\widehat w^\delta_\alpha - v^\delta_\alpha \leq V_\alpha +C\delta^{-1} - c \eta r  \quad \mbox{on} \ \partial B_r.
\end{equation}
It follows from \eqref{rsubball} and \eqref{rsubbdry} that by selecting $r=C/\delta \eta$ for a sufficiently large constant $C>0$, we obtain, for sufficiently small $\delta > 0$,
\begin{equation*}
\delta w_\alpha(y) - \eta + C\gamma \leq 0 \quad \mbox{in} \ B_r \quad \mbox{and} \quad \widehat w^\delta_\alpha - v^\delta_\alpha \leq 0 \quad \mbox{on} \ \partial B_r.
\end{equation*}
The comparison principle (c.f. \cite[Theorem 4.7]{IK}) yields that, for each $\alpha=1,\ldots,\m$,
\begin{equation*}
\widehat w^\delta_\alpha \leq v^\delta_\alpha \quad \mbox{in} \ B_r,
\end{equation*}
and, in particular, $\widehat w^\delta_\alpha(0) \leq v^\delta_\alpha(0)$. Multiplying this last inequality by $\delta$ and sending $\delta \to 0$ gives
\begin{equation*}
- \overline H -C\eta  \leq \liminf_{\delta \to 0} \delta v^\delta_\alpha(0,\omega).
\end{equation*}
Disposing of $\eta > 0$ yields \eqref{cmpineq} for each $\omega\in \widetilde \Omega$.

Since $-\overline H$ is the $L^\infty(\Omega)$ weak-$\ast$ limit of $\delta_j v^{\delta_j}_\alpha(0,\omega)$, the reverse inequality of \eqref{cmpineq} holds and we obtain
\begin{equation} \label{cmpeq}
- \overline H = \liminf_{\delta \to 0} \delta v^\delta_\alpha(0,\omega) \quad \mbox{a.s. in} \  \omega.
\end{equation}
Now an elementary lemma from measure theory (c.f. \cite[Lemma A.6]{AS}) yields that
\begin{equation} \label{convprob}
\delta v^\delta_\alpha(0,\omega) \rightarrow -\overline H \quad \mbox{in probability and in} \ L^1(\Omega,\Prob).
\end{equation}
We now deduce~\eqref{mainstepeq} from a covering argument and the Lipschitz bound \eqref{lipschitz} (see the last step of proof of \cite[Proposition 5.1]{AS}). 
\end{proof}

We next collect some elementary properties of $\overline H$.

\begin{prop}\label{HAM}
The effective Hamiltonian $\overline H:\Rd\times \R_+ \to \R$ has the following properties: 
\begin{enumerate}
\item[(i)] for each $p\in \Rd$, the map $\mu \mapsto H(p,\mu)$ is strictly increasing;
\item[(ii)] for each $\mu \geq 0$, the map $p \mapsto H(p,\mu)$ is convex;
\item[(iii)] there are positive constants $c,C>0$, depending only on the assumptions, such that
\begin{equation} \label{Hbarcoer}
\lambda|p|^2 - C(1+|p|) \leq \overline H(p,\mu) \leq \Lambda |p|^2 + C(|p|+\mu).
\end{equation}
\end{enumerate}
\end{prop}
\begin{proof}
It is immediate from \eqref{CDEinmu} and \eqref{mainstepeq} that, for all $p \in \Rd$ and $0 \leq \mu_1\leq \mu_2$,
\begin{equation*}
c(\mu_2-\mu_1) \leq \overline H(p,\mu_2) - H(p,\mu_1) \leq C(\mu_2-\mu_1)
\end{equation*} 
for $C,c > 0$ depending only on upper bounds for $|p|$ and $\mu_2$. This yields~(i).

To prove (ii), fix $p_1,p_2\in  \Rd$, $\mu \geq 0$, $\omega\in \Omega_0$, set $q := \frac12(p_1+p_2)$ and, for each $\delta > 0$,
\begin{equation} \label{wdeldef}
w^\delta_\alpha(y) : = \frac12v^{\delta}_\alpha(y,\omega;p_1,\mu) + \frac12 v^{\delta}_\alpha(y,\omega;p_2,\mu).
\end{equation}
The convexity of $f_\alpha$ in the differences $z_\alpha-z_\beta$ and the convexity of $H_\alpha$ in $p$ easily yield that $w^\delta$ satisfies
\begin{equation} \label{convderinq}
\delta w^\delta_\alpha - \tr ( A_\alpha\!\left(y, \omega \right)\! D^2 w^\delta_\alpha ) + H_\alpha( \theta+ Dw^\delta_\alpha,y,\omega) + f_\alpha\!\left( w_1^\delta, \ldots, w^\delta_\m, \mu, y, \omega\right)  \leq 0 \quad \mbox{in} \  \Rd.
\end{equation}
Proposition~\ref{vdeltacmp} implies that $w^\delta_\alpha \leq v^\delta_\alpha(y,\omega;q,\mu)$. Multiplying by $-\delta$ and passing to limits with \eqref{mainstepeq} in mind yields (ii).

The bounds \eqref{Hbarcoer} are immediate from \eqref{deltabnd} and \eqref{mainstepeq}.
\end{proof}

\begin{remark}
Notice that \eqref{Hbarcoer} implies that $\overline H(0,0) \leq 0$. It follows, then, from (i), (ii) and (iii), above, that the number $\overline \lambda\geq 0$ given in \eqref{lbar} is well-defined and 
\begin{equation*}
\min_{\Rd} \overline H\big(\cdot,\overline \lambda\big) = 0.
\end{equation*}
\end{remark}

\section{The homogenization of the Hamilton-Jacobi system} \label{SHJthm}

The $L^1$ convergence in the limit \eqref{mainstepeq} can be upgraded to almost sure convergence. That is, we claim that there exists an event $\Omega_2 \subseteq \Omega$ of full probability such that, for every $R> 0$ and $\omega\in \Omega_2$,
\begin{equation}\label{asconv}
\lim_{\delta \to 0}\sup_{y\in B_{R/\delta}} \left| \delta v^\delta_\alpha (y,\omega;p,\mu) + \overline H(p,\mu) \right| = 0.
\end{equation}
To prove this, the subadditive ergodic theorem must be applied to an appropriately chosen subadditive quantity.

Here we outline a proof of \eqref{asconv} which follows closely the ideas of~\cite{AS}. Due to the similarity to~\cite{AS}, we omit the details. In fact, the argument is much simpler here since the system is no more complicated than the scalar case and, unlike~\cite{AS}, we are in the context of a bounded environment. 

\begin{proof}[Sketch of the proof of \eqref{asconv}]
For fixed $p\in \Rd$ and $\mu \geq 0$, we consider what we call the \emph{metric problem}, which is the system of equations
\begin{equation}\label{HJmet}
 - \tr ( A_\alpha\!\left(y, \omega \right)\! D^2 m^\gamma_\alpha ) + H_\alpha( p+ Dm^\gamma_\alpha,y,\omega) + f_\alpha\!\left( m^\gamma_1, \ldots, m^\gamma_\m, \mu, y, \omega\right) = \gamma \quad \mbox{in} \ \Rd \setminus B(x,1),
\end{equation}
coupled with the conditions
\begin{equation}\label{HJmet-bc}
m^\gamma_\alpha(\cdot,x,\omega;p,\mu) = 0 \quad \mbox{on} \ \partial B(x,1) \qquad \mbox{and} \qquad \liminf_{|y|\to \infty} |y|^{-1} m^\gamma_\alpha(y,x,\omega;p,\mu) \geq 0. 
\end{equation}
Here $\gamma\in \R$ is a parameter, and it is possible to show that \eqref{HJmet}-\eqref{HJmet-bc} is well-posed, i.e., there exists a unique solution $m^\gamma_\alpha$ provided that $\gamma > \overline H(p,\mu)$. In fact, for such $\gamma$ there is a comparison principle for the system \eqref{HJmet} in exterior domains under very general growth conditions at infinity (and see Proposition 6.1 in \cite{AS}, which is easily generalized to the weakly coupled system). 

An argument very similar to the proof of~\eqref{collapse} gives the estimate
\begin{equation}\label{collapse-met}
\max_{\alpha,\beta\in \indx} \sup_{\Rd} \big| m^\gamma_\alpha(\cdot,x,\omega) - m^\gamma_\beta(\cdot,x,\omega) \big| \leq C.
\end{equation}
The comparison principle then implies that the $m^\gamma_\alpha$'s are increasing in $\gamma$ and jointly stationary in the sense that, for every $x,y,z\in \Rd$ and $\omega\in\Omega$,
\begin{equation}\label{mgamstat}
m^\gamma_\alpha(y,x,\tau_z\omega) = m^\gamma_\alpha(y+z,x+z,\omega),
\end{equation}
and that, up to a deterministic $C> 0$, the $m^\gamma_\alpha$'s are almost subadditive, i.e., for all $x,y,z\in \Rd$ and $\omega\in\Omega$,
\begin{equation}\label{mgamsl}
m^\gamma_\alpha(y,x,\omega) \leq m^\gamma_\alpha(z,x,\omega) + m^\gamma_\alpha(y,z,\omega) + C.
\end{equation}
The multiparameter subadditive ergodic theorem (c.f.~Akcoglu and Krengel~\cite{AK}) then yields that, almost surely in~$\omega$,
\begin{equation}\label{mbar}
\overline m^\gamma(y-x) = \lim_{t\to\infty} \frac1t m^\gamma_\alpha(ty,tx,\omega)
\end{equation}
for a deterministic function $\overline m^\gamma$ which, due to \eqref{collapse-met}, is independent of $\alpha$. In fact, we can select a single event $\Omega_2\subseteq\Omega$ of full probability on which the limit \eqref{mbar} holds for every $\omega\in \Omega_2$, $p\in \Rd$, $\mu\geq 0$, and $\gamma > \overline H(p,\mu)$. 

With the help of what we have obtained already in Proposition~\ref{mainstep}, we can characterize the limit function $\overline m^\gamma$. We take a subsequence of $\delta$'s along which the convergence in \eqref{mainstepeq} holds almost surely and argue with a reverse perturbed test function argument (introduced in~\cite[Proposition 6.9]{AS}) that 
\begin{equation}\label{limit-met}
\overline H(p+D\overline m^\gamma,\mu) = \gamma \quad \mbox{in} \ \Rd \setminus \{ 0 \}.
\end{equation}
Having identified an almost sure limit in terms of the effective Hamiltonian $\overline H$, we may conclude the proof of \eqref{asconv} for every $\omega\in \Omega_2$ by using a perturbed test function argument very similar to the one in the proof of \cite[Proposition 7.1]{AS} (or the one below).
\end{proof}

With \eqref{asconv} in hand, we present the proof of Theorem~\ref{HJthm}.

\begin{proof}[{Proof of Theorem~\ref{HJthm}}]
We argue only that $u$ is a subsolution of \eqref{HJeff} in $U$, the verfication that it is a supersolution following along similar lines. The proof is by the classical perturbed test function method of Evans~\cite{E}.

Assume that for some $\varphi \in C^\infty(U)$ and $x_0 \in U$,
\begin{equation}  \label{slm}
x\mapsto (u-\varphi)(x) \quad \mbox{has a strict local maximum at} \ x=x_1.
\end{equation}
We must show that 
\begin{equation} \label{wts}
\overline H(D\varphi(x_1),\mu) \leq g(x_1).
\end{equation}
Suppose on the contrary that
\begin{equation} \label{badeta}
\eta : = \overline H(D\varphi(x_1),\mu) - g(x_1) > 0.
\end{equation}
Fix $\omega \in \Omega_0$ for which $\mu_\ep(\omega)\rightarrow \mu$ and $u^\ep(\cdot,\omega) \rightarrow u$ uniformly in a neighborhood of $x_1$. Set $p = D\varphi(x_1)$ and define the perturbed function
\begin{equation*}
\varphi^\ep_\alpha(x): = \varphi(x) + \ep v^\ep_\alpha ( \tfrac x\ep, \omega;p,\mu_\ep),
\end{equation*}
where $v^\ep_\alpha$ is the solution of \eqref{HJaux} with $\delta=\ep$. We claim that, for sufficiently small $r,\ep > 0$,
\begin{equation*}
-\ep \tr\!\left( A_\alpha (\tfrac x\ep,\omega) D^2\varphi^\ep_\alpha \right) + H_\alpha(D\varphi^\ep_\alpha,\tfrac x\ep , \omega ) + f_\alpha \!\left( \frac{\varphi_1^\ep}{\ep}, \ldots, \frac{\varphi_k^\ep}{\ep}, \mu_\ep, \tfrac x\ep, \omega \right) \geq g(x) + \frac12\eta \quad \mbox{in} \ B(x_1,r).
\end{equation*}
Indeed, this follows from the continuity of $H_\alpha$, $f_\alpha$ and $g$, and \eqref{fatrans}, \eqref{HJaux}, \eqref{asconv} and \eqref{badeta}.
The maximum principle for the cooperative system then implies that 
\begin{equation*}
 \max_{\alpha\in\indx}  \max_{\partial B(x_1,r)} \left( u^\ep_\alpha- \varphi^\ep_\alpha \right) =  \max_{\alpha\in\indx} \max_{\overline B(x_1,r)}  \left( u^\ep_\alpha- \varphi^\ep_\alpha \right).
\end{equation*}
Using \eqref{CDEinmu} and \eqref{asconv}, we send $\ep \to0$ to deduce that 
\begin{equation*}
 \max_{\alpha\in\indx}  \max_{\partial B(x_1,r)} \left( u_\alpha- \varphi_\alpha \right) =  \max_{\alpha\in\indx} \max_{\overline B(x_1,r)}  \left( u_\alpha- \varphi_\alpha \right).
\end{equation*}
This contradicts \eqref{slm} for small enough $r> 0$. We have verified \eqref{wts}, which confirms that $u$ is a viscosity subsolution of \eqref{HJeff}. The proof that $u$ is also a supersolution is argued along similar lines.
\end{proof}

\section{Concentration phenomena} \label{proof}

The proof of Theorem~\ref{MAIN} is presented in several steps. First, in the next proposition, we use Egoroff's theorem, the ergodic theorem, and the monotonicity in \eqref{eig-mono} to show that the eigenvalues $\lambda^\ep(U,\omega)$ converge almost surely in $\omega$ to a deterministic limit $\lambda_0$, which is independent of the domain $U$. Comparing the eigenfunctions $\psi^\ep_\alpha$ and the approximate correctors $v^\delta_\alpha$ allows us to conclude that $\lambda_0 = \overline \lambda$, from which the concentration of the eigenfunctions follows easily if $\overline \theta$ can be defined unambiguously. 

\begin{prop} \label{eigslim}
There exist a subset $\Omega_0 \subseteq \Omega$ of full probability and a constant $\lambda_0 \in \R$ such that 
\begin{equation*}
\lim_{\ep \downarrow 0} \lambda^\ep(U,\omega) = \lambda_0 \quad \mbox{for every bounded domain} \ U \subseteq \Rd \ \mbox{and all} \ \omega \in \Omega_0.
\end{equation*}
\end{prop}
\begin{proof}
According to \eqref{eig-mono}, for each fixed $\omega\in \Omega$, the eigenvalue $\lambda^\ep(B_1,\omega)$ is increasing as a function of $\ep$. Therefore, for every $\omega \in \Omega$, there exists a number $\lambda_0(\omega) \in \R$ such that
\begin{equation} \label{lim-omegad}
\lambda^\ep(B_1,\omega) \downarrow \lambda_0(\omega) \quad \mbox{as} \ \ep \downarrow 0.
\end{equation}
We claim that, for each $\mu \in \R$, the event
\begin{equation*}
\Lambda_\mu : = \left\{ \omega \in \Omega : \lambda_0(\omega) \geq \mu \right\}
\end{equation*}
has probability $\Prob[\Lambda_\mu] \in \{ 0,1\}$. This follows from the ergodicity assumption once we show that $\tau_z(\Lambda_\mu) = \Lambda_\mu$ for every $z\in  \Rd$. Indeed, for $|z| \leq \ep^{-1}$, we have, by stationarity,
\begin{equation*}
\lambda^\ep (B_1,\tau_z \omega) = \lambda_1\!\!\left(B(z,\ep^{-1}),\omega\right) \geq \lambda_1\!\!\left( B(0,2\ep^{-1},\omega \right) = \lambda^{\ep/2}(B_1,\omega),
\end{equation*}
and similarly, for such $\ep$, we also have $\lambda^{\ep/2}(B_1,\tau_z\omega) \leq \lambda^\ep(B_1,\omega)$. Hence $\lambda_0(\tau_z\omega) = \lambda_0(\omega)$ for every $z\in  \Rd$. This implies $\tau_z(\Lambda_\mu) = \Lambda_\mu$ for every $z\in \Rd$. 

It is then immediate that \eqref{lim-omegad} may be improved to 
\begin{equation} \label{limaso}
\lambda^\ep(B_1,\omega) \downarrow \lambda_0 \quad \mbox{as} \ \ep \downarrow 0 \quad \mbox{for every} \ \omega \in \Omega_1,
\end{equation}
for some deterministic constant $\lambda_0$ and subset $\Omega_1 \subseteq \Omega$ of full probability.

Using Egoroff's theorem we find a subset $E\subseteq \Omega$ with probability $\Prob[E] \geq \frac12$ such that 
\begin{equation*}
\lambda^\ep(B_1,\omega) \downarrow \lambda_0 \quad \mbox{as} \ \ep \downarrow 0\quad\mbox{uniformly in} \ \omega\in E.
\end{equation*}
For each fixed $\omega\in\Omega$, define the set
\begin{equation*}
A_\omega : = \left\{ y \in  \Rd : \tau_y \omega \in E \right\}.
\end{equation*}
By the ergodic theorem, for each bounded domain $V \subseteq  \Rd$, there exists $\Omega_V \subseteq\Omega$ of full probability such that, for each $\omega \in \Omega_V$, 
\begin{equation*}
\lim_{\ep \to 0} \fint_V \mathds{1}_{A_\omega}\!\left( \frac x\ep \right) \, dx = \Prob[E] \geq \frac12.
\end{equation*}
Choose a countable basis $\mathcal B$ for the Euclidean topology on $ \Rd$ consisting of balls, let $\Omega_2:=\cap_{V\in \mathcal B} \Omega_V$ and define $\Omega_0 : = \Omega_1 \cap \Omega_2$.

Fix now a domain $U \subseteq  \Rd$, a small constant $\eta > 0$, and select an element $V\in \mathcal B$ with $\overline V \subseteq U$ and set $\delta : = \dist(V,\partial U)$. It follows that, for every $\omega \in \Omega_0$, there exists $T_0 = T_0(\omega) > 0$ sufficiently large so that, for all $\omega\in E$ and $0<\ep \leq T_0(\omega)^{-1}$,
\begin{equation*}
\lambda^\ep(B_1,\omega) - \lambda_0 \leq \eta \quad \mbox{and} \quad  \ep A_{\omega} \cap V \neq \emptyset.
\end{equation*}

Now fix $\omega \in \Omega_0$. Suppose that $0<\ep \leq \delta T_0(\omega)^{-1}$ and select $y\in \ep A_{\omega} \cap V$. Then $B(y,\delta) \subseteq U$ and from the stationary hypothesis as well as \eqref{domain-mono}, \eqref{micro-macro} and the above properties, we may deduce that
\begin{align*}
\lambda^\ep(U,\omega) & \leq \lambda^\ep(B(y,\delta),\omega) = \lambda^{\ep/\delta}(B_1,\tau_{\frac y\ep} \omega) \leq \lambda_0 + \eta. 
\end{align*}
It follows that
\begin{equation*}
\limsup_{\ep \downarrow 0} \lambda^\ep(U,\omega) \leq \lambda_0.
\end{equation*}
Owing to the fact that $U \subseteq B_R$ for some large $R> 0$, and that $\omega \in \Omega_0 \subseteq \Omega_1$, we use \eqref{domain-mono} and \eqref{eig-mono} to conclude that 
\begin{equation*}
\lambda^\ep(U,\omega) \geq \lambda^\ep(B_R,\omega) = \lambda^{\ep/R}(B_1,\omega) \geq \lambda_0. \qedhere
\end{equation*}
\end{proof}

Next we use Theorem~\ref{HJthm} to show that $\lambda_0$ equals $\overline\lambda$ defined in \eqref{lbar}. From this we conclude the concentration \eqref{concent} of the eigenfunctions  and complete the proof of our main theorem.

\begin{proof}[{Proof of Theorem~\ref{MAIN}}]
Let $\psi^\ep_\alpha$ denote the functions defined in \eqref{hopf-cole} for $\alpha=1,\ldots,\m$ and normalized according to $\psi^\ep_1(x_0,\omega) = 0$ for some fixed $x_0 \in U$. An argument very similar to the one in the proof of Lemma~\ref{vdelta} yields,  for each $V \subset\subset U$, the bound
\begin{equation*}
\sup_{\alpha,\beta \in\indx} \sup_{ V} | \psi^\ep_\alpha(\cdot,\omega) - \psi^\ep_\beta(\cdot,\omega) | \leq C\ep,
\end{equation*}
and then the local Lipschitz estimates
\begin{equation*}
\sup_{\alpha\in\indx} \sup_{V} |D\psi^\ep_\alpha(\cdot ,\omega)| \leq C,
\end{equation*}
for a  $C>0$ independent of $\ep$. Taking a subsequence, also denoted by $\ep$, we find $\psi\in C^{0,1}_{\mathrm{loc}}(U)$ such that, as $\ep \to 0$ and for every $\alpha = 1,\ldots,\m$,
\begin{equation}\label{}
\psi^\ep_\alpha \rightarrow \psi \quad \mbox{locally uniformly in} \ U.
\end{equation}
Now Theorem~\ref{HJthm} and Proposition~\ref{eigslim} imply that $\psi$ satisfies the equation
\begin{equation*}
\overline H(D\psi,\lambda_0) = 0 \quad \mbox{in} \ U.
\end{equation*}
It follows at once that $\lambda_0 \leq \overline \lambda$.

To obtain the reverse inequality, we select $(p,\mu)$ such that $\overline H(p,\mu) < 0$. Set $\delta > 0$ sufficiently small so that the event $\inf_{U/\delta} \delta v^\delta_\alpha(\cdot,\omega;p,\mu) > 0$ has probability at least $\frac12$. For  $\omega$ belonging to this event, and, if in addition $\lambda_\ep(U,\omega) \leq \mu$,  the map
\begin{equation*}
x\mapsto \min_{\alpha\in\indx} \left( \psi^\delta_\alpha(x,\omega) - v^{\delta}_\alpha(\delta x,\omega;p,\mu) \right)
\end{equation*}
cannot have a local minimum in $U$ according to the comparison principle. This is a contradiction, since $\psi^\ep_\alpha(x,\omega) \to +\infty$ as $x\to \partial U$ and $v^{\delta}_\alpha(\cdot,\omega;p,\mu)$ is bounded. We deduce that $\Prob[ \lambda_\delta(U,\omega) > \mu ] \geq \frac12$ for small $\delta > 0$. According to Proposition~\ref{eigslim} we have $\lambda_0 \geq \mu$, and hence $\lambda_0 \geq \overline \lambda$. Therefore $\lambda_0 = \overline \lambda$ and we obtain the limit \eqref{drift}.

Finally, in the case $\{ p : \overline H(p,\overline \lambda) = 0 \} = \{\overline \theta\}$, we obtain that $D\psi =\overline \theta$ almost everywhere in $U$. It follows that $\psi(x) =\overline \theta\cdot (x-x_0)$ for each $x\in U$, and hence the full sequence $\psi^\ep_\alpha$ converges to $\psi$. The concentration behavior \eqref{concent} then follows.
\end{proof}

\section{Strict convexity of $p\mapsto \overline H(p,\mu)$ in uniquely ergodic environments} \label{SC}

We prove Theorem~\ref{uesc}. Throughout this section, we assume that the action of $( \tau_y )_{y\in\Rd}$ on the environment $(\Omega,\mathcal F, \Prob)$ is uniquely ergodic (see Definition~\ref{uedef}).

\medskip

It is worth revisiting the proof of the convexity of $p\mapsto \overline H(p,\mu)$ (Proposition~\ref{HAM}(ii)) to see if there is extra information we discarded. The argument essentially comes down to the derivation of \eqref{convderinq}. There is no doubt that any strict convexity on the part of $\overline H$ must be inherited from the $H_\alpha$'s, which satisfy, for every $p_1,p_2\in \Rd$ with $q := \frac12 p_1 + \frac12 p_2$ and $(y,\omega) \in \Rd\times\Omega$,
\begin{align} 
\frac12 H_\alpha(p_1,y,\omega)  + \frac12 H_\alpha(p_2,y,\omega) - H(q,y,\omega) & = \frac14 (p_1-p_2) \cdot A(y,\omega) (p_1-p_2) \label{HalphSC} \\
& \geq \frac14 \lambda |p_1-p_2|^2.\nonumber
\end{align}
Using \eqref{HalphSC}, we observe that, with $w_\alpha^\delta$ defined as in \eqref{wdeldef}, we may improve \eqref{convderinq}  to 
\begin{multline} \label{convderinqimp}
\delta w^\delta_\alpha - \tr ( A_\alpha\!\left(y, \omega \right)\! D^2 w^\delta_\alpha ) + H_\alpha( q+ Dw^\delta_\alpha,y,\omega) + f_\alpha\!\left( w_1^\delta, \ldots, w^\delta_\m, \mu, y, \omega\right)  \\ \leq -\frac14\lambda\big| p_1 - p_2 + Dv^\delta_\alpha(y,\omega;p_1,\mu) - Dv^\delta_\alpha(y,\omega;p_2,\mu) \big|^2 =: -h_\alpha(y,\omega) \quad \mbox{in} \  \Rd.
\end{multline}
The hope is to use the term $-h_\alpha$ on the right side of \eqref{convderinqimp} to show that, for some $c> 0$, 
\begin{equation*}\label{}
w^\delta_\alpha - v^\delta_\alpha(y,\omega;\theta,\mu) \leq -c\delta^{-1}.
\end{equation*}
If $h_\alpha$ is bounded below by a positive constant, then the desired conclusion is immediate. However, for $p_1$ close to $p_2$, there is not a definite reason why this should be true. All we can say is that $h$ is stationary, nonnegative, bounded, and satisfies, by Jensen's inequality and \eqref{EDv0},
\begin{equation} \label{Ehpos}
\E h_\alpha(0,\cdot) \geq \frac14 \lambda |p_1-p_2|^2.
\end{equation}
However, what we actually need is something weaker than for $h_\alpha$ to be bounded below by a positive constant. As we will see, it is enough to rule out the presence of large ``bare spots." That is, we need to ensure that, for some $R>0$, the set on which $h_\alpha$ is greater than some positive constant takes a uniform proportion of each ball of radius $R$. This is precisely what the unique ergodicity hypothesis gives us, as we see in the following lemma.

\begin{lem} \label{ueappl}
Suppose that $g = g(y,\omega)$ is stationary, nonnegative, and does not vanish a.s. in $\omega$. Then there exist constants $\eta,\rho > 0$, depending on the distribution of $g(0,\cdot)$, and a subset $\Omega_1 \subseteq\Omega$ of full probability, such that, for every $\omega\in \Omega_1$, there exists $R> 0$ such that
\begin{equation} \label{UEder}
\inf_{z\in \Rd} \big| \{ y\in B(z,R) : g(y) \geq \eta \} \big| \geq \rho |B_R|.
\end{equation}
\end{lem}
\begin{proof}
Let
\begin{equation*}
E: = \Big\{ \omega \in \Omega \,:\, \big| \{ y\in B(0,1) : g(y,\omega) \geq \eta \}\big| \geq \rho |B_1|  \Big\},
\end{equation*}
with $\eta,\rho >0$ chosen small enough so that $\Prob[E] > 0$. According to \eqref{uecond}, there exists a subset $\widetilde \Omega \subseteq \Omega$ of full probability such that, for every $\omega \in \widetilde \Omega$, there exists $R> 1$ sufficiently large that 
\begin{equation*}
\inf_{z\in \Rd} \ \fint_{B(z,R)} \mathds{1}_{E} (\tau_y\omega) \, dy \geq \frac12 \Prob[E] > 0.
\end{equation*}
That is, for each $\omega \in \widetilde \Omega$, there exists $R>1$, depending on $\omega$, such that
\begin{equation} \label{bigDz}
\big| \{ y \in B(z,R) : \tau_y\omega \in E \} \big| \geq c_1 |B_R|,
\end{equation}
with $c_1:= \frac12 \Prob[E]>0$. Freeze $\omega\in \widetilde \Omega$ for the remainder of the argument, let 
\begin{equation}\label{}
D(z): =  \{ y \in B(z,R) : \tau_y\omega \in E \}
\end{equation}
and observe that the stationarity of $g$ yields
\begin{equation*}
D(z) = \Big\{ y\in B(z,R): \big| \{ x\in B(y,1) : g(x,\omega) \geq \eta \}\big| \geq \rho  \Big\}.
\end{equation*}
According to the Vitali covering lemma, there exist $y_1,\ldots,y_\ell \in D(z)$ such that the balls $\{ B(y_i,1) \}_{i=1}^\ell$ are disjoint and
\begin{equation} \label{covDz}
D(z) \subseteq \bigcup_{i=1}^\ell B(y_i,3).
\end{equation}
Since the balls $\{ B(y_i,1) \}_{i=1}^\ell$ are disjoint and $\tau_{y_i} \omega \in E$, we have, for any $z\in \Rd$,
\begin{equation*}
\Big| \big\{ y \in B(z,2R) : g(y,\omega) \geq \eta \big\} \Big| \geq \Big| \bigcup_{i=1}^\ell \{ y \in B(y_i,1) : g(y,\omega) \geq \eta \} \Big|  \geq \rho \ell |B_1|.
\end{equation*}
It follows from \eqref{bigDz} and \eqref{covDz} that $\ell |B_1| \geq c |D(r)| \geq c c_1 |B_R| \geq c|B_{2R}|$ and hence  \eqref{UEder}.
\end{proof}

In a uniquely ergodic environment, we can prove that many limits derived from the ergodic theorem are uniform. Another example is the following useful lemma, which generalizes the existence of ``approximate correctors" used by Ishii~\cite{I2} to prove homogenization of Hamilton-Jacobi equations in an almost periodic environment (the proof we give below also works in the non-viscous setting for first-order Hamilton-Jacobi equations). 

\begin{lem}\label{unifconv}
Under the uniquely ergodic assumption, the convergence \eqref{asconv} can be improved to
\begin{equation}\label{uniqunf}
\delta v^\delta_\alpha(y,\omega;p,\mu) \rightarrow \overline H(p,\mu) \quad \mbox{uniformly in} \ \Rd \ \mbox{and a.s. in} \ \omega.
\end{equation}
\end{lem}
\begin{proof}
Let $\ep > 0$. The unique ergodicity assumption and \eqref{mainstepeq} yield some $\delta_1> 0$ small and $R > 1$ large such that, for all $0< \delta < \delta_0$, 
\begin{equation*}
\sup_{\alpha\in \indx} \sup_{y\in \Rd} \inf_{B(y,R)} \left| \delta v^\delta_\alpha(y,\omega;p,\mu) + \overline H(p,\mu) \right| \leq \ep. 
\end{equation*}
But then the Lipschitz bound \eqref{lipschitz} gives, for every $0 < \delta < \delta_0\min\{1, (CR)^{-1}\}$,
\begin{equation}\label{}
\sup_{\alpha\in \indx} \sup_{y\in \Rd}\left| \delta v^\delta_\alpha(y,\omega;p,\mu) + \overline H(p,\mu) \right| \leq \ep + \delta CR \leq 2\ep. \qedhere
\end{equation}
\end{proof}

Next we discuss a so-called ``growth lemma," an important analytic tool in the proof of Theorem~\ref{uesc}. It is a quantitative strong maximum principle which measures how the negative term $-h_\alpha$ on the right side of \eqref{convderinqimp} forces $w^\delta_\alpha$ to be lower in comparison to $v^\delta_\alpha(\cdot,\omega;\theta,\mu)$. We do not give the proof here, since going into details would take us very far off course. However, the proof is nearly the same as the proof of the classical growth lemma (c.f. Theorem 2 on page 118 of Krylov~\cite{Kry}), which follows from the ABP inequality. For this purpose we need the following ABP inequality for weakly coupled elliptic systems proved by Busca and Sirakov~\cite{BS}. 

\begin{lem} \label{growthlem}
Fix $\mu \geq 0$, $p\in \Rd$, $\omega \in \Omega$ and suppose that $\sigma,\tau \in \R$ and $u_\alpha$ and $v_\alpha$ satisfy
\begin{equation}
- \tr ( A_\alpha D^2 u_\alpha ) + H_\alpha( p+ Du_\alpha,y,\omega) + f_\alpha(u_1,\ldots,u_\m,\mu,y,\omega) \leq  \tau + \sigma - h_\alpha(y,\omega) \quad \mbox{in} \ B_{2R}
\end{equation}
and
\begin{equation}
- \tr ( A_\alpha D^2 v_\alpha ) + H_\alpha( p+ Dv_\alpha,y,\omega) + f_\alpha(v_1,\ldots,v_\m,\mu,y,\omega) \geq  \tau  \quad \mbox{in} \ B_{2R},
\end{equation}
where for some $\eta, \rho > 0$,
\begin{equation}
\min_{\alpha\in \indx} \big| \{ x\in B_R : h_\alpha(x,\omega) \geq \eta \} \big| \geq \rho.
\end{equation}
Then there exist constants $\kappa, \sigma_0 > 0$, depending on the constants in the assumptions as well as $R$, $\eta$ and $\rho$, such that $\sigma \leq \sigma_0$ implies that
\begin{equation*}
\min_{\alpha\in\indx} \inf_{B_{R}} (v_\alpha-u_\alpha) \geq \kappa + \min_{\alpha\in\indx} \inf_{B_{2R}} (v_\alpha-u_\alpha).
\end{equation*}
\end{lem}

We combine the preceding Lemmata into a proof of Theorem~\ref{uesc}.

\begin{proof}[{Proof of Theorem~\ref{uesc}}]
We select $\mu \geq 0$, $p_1 \neq p_2$ and set $q: = \frac12(p_1+p_2)$. We argue by contradiction under the false assumption that $\overline H(p_1,\mu) + H(p_2,\mu) = 2 H(q,\mu)$, proceeding by way of a comparison between the functions $w^\delta_\alpha$ defined in \eqref{wdeldef} and the solutions $v^\delta_\alpha(\cdot,\cdot;q,\mu)$ of \eqref{HJaux} with $p=q$.

Fix $\ep > 0$ very small and $R > 0$ very large. According to Lemma~\ref{unifconv}, we may choose $\delta > 0$ sufficiently small to ensure that, for each $\alpha=1,\ldots,\m$,
\begin{equation}\label{}
\delta v^\delta_\alpha (y,\omega;q,\mu) \leq - \overline H(q,\mu)  + \ep  \quad \mbox{in} \ \Rd
\end{equation}
as well as
\begin{equation}\label{}
\delta w^\delta_\alpha(\cdot,\omega) \geq -\overline H(q,\mu) - \ep \quad \mbox{in} \ \Rd.
\end{equation}
Therefore, we have 
\begin{equation}\label{}
- \tr \big( A_\alpha D^2v^\delta_\alpha \big) + H_\alpha\big( q+ Dv^\delta_\alpha,y,\omega\big) + f_\alpha \big(v^\delta_1,\ldots,v^\delta_\m,\mu,y,\omega\big) \geq \overline H(p,\mu) - \ep  
\end{equation}
and, for $h_\alpha$ defined in \eqref{convderinqimp}, 
\begin{equation}\label{}
- \tr \big( A_\alpha D^2v^\delta_\alpha \big) + H_\alpha\big( q+ Dv^\delta_\alpha,y,\omega\big) + f_\alpha \big(v^\delta_1,\ldots,v^\delta_\m,\mu,y,\omega\big) \leq \overline H(p,\mu) + \ep -h_\alpha(y,\omega).
\end{equation}
According to Lemma~\ref{ueappl} and the growth lemma, if we choose $\ep > 0$ sufficiently small, then 
\begin{equation}\label{defMbt}
M(y):= \min_{\alpha\in\indx} \big( v^\delta_\alpha(y) - w^\delta_\alpha(y) \big)
\end{equation}
satisfies, for some $\kappa, R > 0$, 
\begin{equation}\label{caught}
M(y) \geq \kappa + \inf_{z\in B(y,R)} M(z).
\end{equation}
Such a function cannot be bounded. Indeed, if $M$ were bounded, then, for any $\beta> 0$, the function $z\mapsto M(z) + \beta |z|$ would achieve its global minimum at some point $y \in \Rd$. But then we would have 
\begin{equation*}\label{}
M(y) \leq M(z) + \beta (|z|-|y|) \quad \mbox{for all} \ z \in \Rd,
\end{equation*}
which is incompatible with \eqref{caught} if we take $\beta < \kappa / R$. We conclude that $M$ is unbounded. However, in light of its definition \eqref{defMbt}, the unboundedness of $M$ contradicts the boundedness of $v^\delta_\alpha$ and $w^\delta_\alpha$. This completes the proof.
\end{proof}

\subsection*{Acknowledgements}
The first author was partially supported by NSF Grant DMS-1004645, and the second author by NSF Grant DMS-0901802. We thank Amie Wilkinson and Albert Fathi for helpful comments.

\small

\bibliographystyle{plain}
\bibliography{neutrons}

\end{document}